\documentclass[12pt]{amsart}
\usepackage{amssymb}
\usepackage[all,cmtip]{xy}
\usepackage{fullpage}
\usepackage{graphicx}
\usepackage[usenames]{color}
\usepackage{etex,hyperref}
\usepackage{mathrsfs}
\usepackage{stmaryrd}
\usepackage{tikz}

\allowdisplaybreaks

\newcommand{\Cl}{\operatorname{\sf Cl}}
\newcommand{\da}{\downarrow}
\newcommand{\End}{\operatorname{\sf End}}
\newcommand{\ep}{\varepsilon}

\renewcommand{\ge}{\geqslant}
\newcommand{\Hom}{\operatorname{\sf Hom}}

\newcommand{\Ker}{\operatorname{\sf Ker}}
\renewcommand{\le}{\leqslant}
\newcommand{\Mat}{\mathsf{Mat}}

\newcommand{\Spec}{\operatorname{\sf Spec}}

\newcommand{\llb}{\llbracket}
\newcommand{\rrb}{\rrbracket}
\newcommand{\sym}[1]{\mathfrak{S}_{#1}}

\newcommand{\Pic}{\mathsf{Pic}}

\newcommand{\Inv}{\mathsf{Inv}}

\newcommand{\bA}{\mathbb{A}}

\newcommand{\bG}{\mathbb{G}}

\newcommand{\bX}{\mathbb{X}}
\newcommand{\bZ}{{\mathbb{Z}}}
\newcommand{\mcA}{\mathcal{A}}

\newcommand{\mcN}{\mathcal{N}}

\newcommand{\mfm}{\mathfrak{m}}
\newcommand{\mfp}{\mathfrak{p}}

\newcommand{\op}{{\mathrm{op}}}

\numberwithin{equation}{section}
\counterwithin{figure}{section}
\newtheorem{theorem}[equation]{Theorem}
\newtheorem{lemma}[equation]{Lemma}
\newtheorem{corollary}[equation]{Corollary}
\newtheorem{proposition}[equation]{Proposition}
\theoremstyle{definition}
\newtheorem{example}[equation]{Example}
\newtheorem{conjecture}[equation]{Conjecture}
\newtheorem{notation}[equation]{Notation}
\newtheorem{remark}[equation]{Remark}
\newtheorem{definition}[equation]{Definition}

\newif\ifGroebner
\Groebnerfalse

\title[Affine nilCoxeter algebras]{Simple modules for affine nilCoxeter algebras}
\author{David J. Benson}
\address[D. J. Benson]{Institute of Mathematics, University of Aberdeen, Aberdeen AB24 3UE,
United Kingdom.}
\email{d.j.benson@abdn.ac.uk}

\author{Kay Jin Lim}
\address[K. J. Lim]{Division of Mathematical Sciences, Nanyang
  Technological University, SPMS-04-01, 21 Nanyang Link, Singapore
  637371.} 
\email{limkj@ntu.edu.sg}

\begin{document}

\begin{abstract}
We study the representation theory of the affine nilCoxeter algebra
$A$ of type $\tilde A_{n-1}$, over a field $k$ of any characteristic. 
Our main theorem states that this is a
Noetherian prime affine PI algebra of PI degree $n!$. As a
consequence, the simple $A$-modules are all finite dimensional, and
the maximum dimension of a simple module is $n!$ over a suitable
finite extension of $k$.

To achieve this, we
investigate a large commutative subalgebra $C$ which is finitely
generated as an algebra and over which $A$ is finitely generated as a
module. We
show that the associated primes of $C$ are minimal primes, and there
are $n!$ of them, regularly permuted by $\sym{n}$. The algebra
$R=C^{\sym{n}}$ is equal to the centre of $A$, and isomorphic to $C/\mfp$ for each of the
minimal primes $\mfp$. 

We prove that the ring $R$ is isomorphic to
$k[X_1,\dots,X_{n-1}]^{\mu_n}$, where $\mu_n$ is the finite group
scheme of $n$th roots of unity, acting so that $X_i$ has degree $i$
modulo $n$. The ring $R$ is Cohen--Macaulay,
and is Gorenstein if and only if $n$ is odd or $n=2$. It is a toric ring, with
divisor class group $\Cl(R)\cong\bZ/n$, and every projective
$R$-module is free. 
\end{abstract}

\keywords{Affine nilCoxeter algebra, affine PI algebra, affine Weyl group, simple module,
symmetric group}

\subjclass[2020]{16S99 (primary), 05E10, 16R20, 20F55 (secondary)}
\maketitle

\section{Introduction}

The affine nilCoxeter algebra $\mcN(\tilde A_{n-1})$ has been studied
in~\cite{Bandlow/Schilling/Zabrocki:2011a,
Berg/Bergeron/Pon/Zabrocki:2012a,
Berg/Bergeron/Thomas/Zabrocki:2012a,
Berg/Saliola/Serrano:2014a,
Lam:2006a,
Lam:2008a,
Lee:2017a}. 
It is described by generators and relations
corresponding to the associated Coxeter group $W(\tilde
A_{n-1})\cong\bZ^{n-1}\rtimes\sym{n}$, except that instead of the
generators squaring to one, they square to zero (see
Section~\ref{se:N} for details). Our purpose here is to investigate the
representation theory of this algebra. In particular, we show that
the simple modules are all finite dimensional, and have central
characters associated with them. For each central character, there is
an associated finite dimensional algebra whose representation theory
is equivalent to the representations with that central character. For
a dense open subset of these, the algebra is simple.
Our main theorem is as follows.

\begin{theorem}
Let $A$ be an affine nilCoxeter algebra of type $\tilde A_{n-1}$. 
\begin{enumerate}
\item $A$ is a Noetherian affine PI algebra.
\item Every simple $A$-module is finite dimensional. 
\item $A$ is prime.
\item If $k$
is algebraically closed, then the 
maximum dimension of a simple module is $n!$. 
\end{enumerate}
\end{theorem}

Part (1) is proved in Corollary~\ref{co:PI}, and part (2)
is a consequence, Theorem~\ref{th:fd}. Part (3) is proved in
Theorem~\ref{th:prime}.
We then investigate the representation theory of $A$ by introducing a
large commutative subalgebra $C$ that admits an action of
$\sym{n}$.

\begin{theorem}
The group $\sym{n}$ permutes the minimal primes of $C$ regularly, each
one corresponding to a total order of the set $\{1,\dots,n\}$.
The algebra $R=C^{\sym{n}}$ is the centre of $A$. For each minimal
prime $\mfp$ of $C$ there is a canonical isomorphism $C/\mfp\cong R$.
\end{theorem}

\begin{theorem}
The centre $R=C^{\sym{n}}$ of $A$ is isomorphic to
$k[X_1,\dots,X_{n-1}]^{\mu_n}$, the fixed points of the finite group
scheme $\mu_n$ on the polynomial ring. It is 
a normal integral domain of Krull
dimension $n-1$. It is Cohen--Macaulay, and it is Gorenstein if and
only if $n$ is odd or $n=2$. For $n\ge 3$, it is a
toric ring with divisor class group $\Cl(R)\cong \bZ/n$, every projective $R$-module is free, and Picard group is equal to zero.
\end{theorem}

For the proof of this theorem, see Section~\ref{se:R}.\bigskip\pagebreak[3]

Given a simple $A$-module, by Schur's lemma the action of the centre
$R$ factors through a field $\tilde k=R/\mfm R$, where $\mfm$ is a
maximal ideal and $\tilde k$ is a finite extension of $k$.
We show that for a dense open subset of the closed points $\mfm$ in
$\Spec(R)$, there is just one simple module for $A/\mfm A$, and its
dimension is $n!$ over $\tilde k$. Thus we have 
\[ A/\mfm A\cong\Mat_{n!}(\tilde k). \]
This is the maximum dimension of a simple $A$-module.

We give more complete information in the cases of $\tilde A_1$ and
$\tilde A_2$. In the case of $\tilde A_1$ the simple modules have
dimension one and two over $\tilde k$, while in the case of $\tilde
A_2$ they have dimensions one, three and six. For details, see
Sections~\ref{se:A1tilde} and~\ref{se:A2tilde}.\bigskip

The structure of the paper is as follows. In Section~\ref{se:W}, we
discuss the affine Weyl group of type $\tilde A_{n-1}$. We prove 
Theorem~\ref{th:ww''w'},
which is used in the proof that $\mcN(\tilde
A_{n-1})$ is prime. This seems to be an interesting new theorem about
affine Weyl groups, and is as follows.

\begin{theorem}
Given $w$ and $w'\in W(\tilde A_{n-1})$, there exists $w''$ such that
the length of $ww''w'$ is the sum of the lengths of $w$, $w''$ and $w'$.
\end{theorem}

In Section~\ref{se:N} we introduce the affine nilCoxeter algebra, and
in Section~\ref{se:C} we discuss a large commutative subalgebra $C$
whose characters will give us a theory of weights for $\mcN(\tilde
A_{n-1})$-modules. This admits an action of $\sym{n}$ that regularly
permutes its minimal primes, and in Section~\ref{se:centre} we show that
$R=C^{\sym{n}}$ is the centre of $\mcN(\tilde A_{n-1})$. The ring
theory for $R$ is discussed in Section~\ref{se:R}. In
Section~\ref{se:Amodules} we develop the theory of weights for
$\mcN(\tilde A_{n-1})$-modules, and use this to prove that it is a
prime affine PI algebra of PI degree $n!$. Finally, in the last few
sections we give some examples of small rank to illustrate the theory.
These computations lead us to the following conjecture,
which we verify in the cases of $\tilde A_{n-1}$ for $n=2$, $3$ and $4$.

\begin{conjecture}\label{conj:multinomial}
Let $A$ be a nilCoxeter algebra of type $\tilde A_{n-1}$ with $n\ge 2$. 
Over an algebraically closed
field of characteristic zero, the weight spaces of
simple $A$-modules are one dimensional. This implies that the
dimension of a simple module is a
multinomial coefficient 
$\binom{n}{n_1\dots n_s}=\frac{n!}{n_1!\dots n_s!}$, 
where the stabilisers of the 
corresponding weight spaces  are the conjugates of the Young subgroup
$\sym{n_1}\times\dots\times\sym{n_s}$, see Lemma~\ref{le:stab(p)}.
\end{conjecture}

\noindent
{\bf Acknowledgements.} We would like to thank Matthew Dyer and Toby
Stafford for helpful email exchanges related to this work. We would
also like to acknowledge that we performed many helpful computations using the
computational algebra system {\sc Magma}~\cite{Bosma/Cannon/Playoust:1997a}
in the course of this work. We also acknowledge ChatGPT Pro for finding (minor) inaccuracies in a previous version of the paper. The second author is supported by Singapore Ministry of Education AcRF Tier 2 grant MOE-T2EP20225-0003.

\section{\texorpdfstring{The affine Weyl group $W(\tilde A_{n-1})$}
{The affine Weyl group W(Ãₙ₋₁)}}\label{se:W}

In this section, we summarise the properties of the affine Weyl group
that we will be using.
Let $W(\tilde A_{n-1})$ be the affine Weyl group of type $\tilde
A_{n-1}$. 
This has generators
$w_i$, $0\le i\le n-1$, and relations $w_i^2=1$,
$w_iw_{i+1}w_i=w_{i+1}w_iw_{i+1}$, $w_iw_j=w_jw_i$ if $|i-j|\ge 2$,
and where the indices in these relations are to be read modulo
$n$. A reference for affine reflection groups is 
Humphreys~\cite{Humphreys:1990a}. 

Lusztig~\cite[\S3.6]{Lusztig:1983a} introduced the 
permutation representation of $W(\tilde A_{n-1})$ on $\bZ$, as
those permutations $w\colon\bZ\to\bZ$ with the properties that 
\begin{enumerate}
\item[(L.1)] $w(i+n)=w(i)+n$
\item[(L.2)] $\sum_{i=1}^{n}( w(i)-i)=0$.
\end{enumerate}
This permutation representation is also studied in
Shi~\cite[chapter~4]{Shi:1986a}, 
Bj\"orner and Brenti~\cite{Bjorner/Brenti:1996a}, 
Ehrenborg and Readdy~\cite[\S4]{Ehrenborg/Readdy:1996a}, 
Eriksson and Eriksson~\cite{Eriksson/Eriksson:1998a}, 
Xi~\cite[\S2.1]{Xi:2002a},
Clark and Ehrenborg~\cite{Clark/Ehrenborg:2011a}.
In this permutation representation, we have $w_i\mapsto (i+jn,i+jn+1)$ for  
$0\le i\le n-1$.  
\emph{Window notation} for an element $w\in W(\tilde A_{n-1})$ is
$w = [w(1),w(2),\dots,w(n)]$. So the generator
$w_i$ in window notation is $[1,2,\dots,i-1,i+1,i,i+2,\dots,n]$ for
$1\le i\le n-1$, and $w_0=[0,2,3,\dots,n-1,n+1]$.

The \emph{length} of an element $w\in W(\tilde A_{n-1})$ is the
shortest length of a word in the generators $w_i$ representing
$w$. 
A \emph{minimal word} for $w$ is a word of length $\ell(w)$
representing $w$.

\begin{definition}
For $1\le i\le n$, we define  $\Inv_i(w)$ to be the number of $j>i$ such  
that $w(j)<w(i)$ in the permutation representation on $\bZ$.  
\end{definition}

\begin{lemma}\label{le:ell1}
The length of $w\in W(\tilde A_{n-1})$ is equal to
\begin{align*} 
\ell(w)
&=\sum_{i=1}^n \Inv_i(w)
=|\{(i,j)\in\bZ^2\mid 1\le i\le  n,\ j>i\text{ and }
w(j)<w(i)\}|.
\end{align*}
\end{lemma}
\begin{proof}
This is~\cite[Theorem~15]{Eriksson/Eriksson:1998a}, 
\cite[Proposition~4.1\,(ii)]{Bjorner/Brenti:1996a},
\cite[\S2.1.3]{Xi:2002a}.
\end{proof}

\begin{definition}
The \emph{affine inversion table} of $w\in W(\tilde A_n)$ is 
\[ \Inv(w) = \llb\Inv_1(w),\dots,\Inv_n(w)\rrb. \]
\end{definition}

Note that, taking the $i\in\{1,\dots,n\}$ for which $w(i)$ is
minimal, we have $\Inv_i(w)=0$. So $\Inv(w)$ has to have some 
entry equal to zero.

\begin{theorem}\label{th:bijection}
The map $w\mapsto \Inv(w)$ gives a bijection between $W(\tilde A_n)$
and $(\bZ_{\ge 0})^n\setminus(\bZ_{>0})^n$.
\end{theorem}
\begin{proof}
This is~\cite[Theorem~4.6]{Bjorner/Brenti:1996a}.
\end{proof}

\begin{corollary}\label{co:Poincareseries}
We have 
\[ \sum_{w\in W(\tilde A_{n-1})}t^{\ell(w)}=\frac{1-t^n}{(1-t)^n}. \]
\end{corollary}
\begin{proof}
This follows by combining Theorem~\ref{th:bijection} and Lemma~\ref{le:ell1},
see~\cite[Corollary~4.7]{Bjorner/Brenti:1996a},
\cite[Corollary~4.3]{Ehrenborg/Readdy:1996a}.
\end{proof}

\begin{lemma}\label{le:ell3}
The length of $w\in W(\tilde A_{n-1})$ is given by
\begin{equation*}
\ell(w) = \sum_{1\le i<j\le
    n}\left|\left\lfloor\frac{w(j)-w(i)}{n}\right\rfloor\right|, 
\end{equation*}
Here, 
$\lfloor\ \rfloor$ denotes integer part, so for example 
$\lfloor-\frac{5}{4}\rfloor=-2$.
\end{lemma}
\begin{proof}
See~\cite[Lemma~4.2.2]{Shi:1986a},
\cite[Proposition~3.1]{Bjorner/Brenti:1996a}. 
\end{proof}

If we ignore condition (L.2) at the beginning of this section, and
just use (L.1),
the group of permutations of $\bZ$ defined this way is the extended
affine Weyl group 
$\tilde W(\tilde A_{n-1})$, which sits in a short exact sequence
\[ 1 \to W(\tilde A_{n-1})\to \tilde W(\tilde A_{n-1}) \to\bZ\to 1, \]
where the map to $\bZ$ is 
\[ w\mapsto \frac{1}{n}\sum_{i=1}^n(w(i)-i)= 
\frac{1}{n}\sum_{i=1}^{n} w(i)\ -\ \frac{n+1}{2}. \]
The group $\tilde W(\tilde A_{n-1})$ is a semidirect product
$\bZ^{n}\rtimes\sym{n}$ where 
\[ \sym{n}=\langle w_1,\dots,w_{n-1}\rangle=W(A_{n-1}), \] 
and $\bZ^n$ is the
natural permutation module for $\sym{n}$, acting via
\begin{equation}\label{eq:action}
\sigma(r_1,\dots,r_n)=(r_{\sigma^{-1}(1)},\dots,r_{\sigma^{-1}(n)}). 
\end{equation} The map to the quotient
$\sym{n}$ of $\bZ^n\rtimes\sym{n}$ sends $w_i$ to $(i,i+1)$ for $1\le i\le n-1$ and $w_0$ to $(1,n)$. 
In terms of the above permutation representation, this amounts to   
identifying $i$ with $i+n$ so that the action is on $\bZ/n$.  
 The $j$th coordinate in $\bZ^n$ corresponds to the permutation sending $i\in\bZ$ to $i+n$ if
$i-j$ is divisible by $n$ and fixing $i$ otherwise. This is the vector
$(0,\dots,0,1,0,\dots,0)\in\bZ^n$ with the $1$ in the $i$th place.

The
subgroup $W(\tilde A_{n-1})$ of $\tilde W(\tilde A_{n-1})$ sits in a
split short exact sequence
\begin{equation}\label{eq:ZWS}
1\to \bZ^{n-1} \to W(\tilde A_{n-1}) \to \sym{n} \to 1
\end{equation}
so that we have $W(\tilde A_{n-1})\cong\bZ^{n-1}\rtimes\sym{n}$. Here  
$\bZ^{n-1}$ is the subgroup of $\bZ^n$ given by $n$-tuples with sum
zero, and the action of $\sym{n}$ is by permutations of the
coordinates. There is a splitting which sends $\sym{n}$ to
\begin{equation}\label{eq:WAn-1}
 W(A_{n-1})=\langle w_1,\dots,w_{n-1}\rangle\subseteq W(\tilde
  A_{n-1}). 
\end{equation}
In terms of this semidirect product, we have
\begin{align*} 
w_i& \mapsto (0,0,\dots,0).(i,i+1)& (1\le i\le n-1),\\ 
w_0&\mapsto (1,0,\dots,0,-1).(1,n). 
\end{align*}

\begin{definition}\label{def:wbar}
For $w\in W(\tilde A_{n-1})$, we write $\bar w$ for its image in
$\sym{n}$ under the map~\eqref{eq:ZWS}. So for example $\bar
w_0=(1,n)$, and for $1\le i\le n-1$ we have $\bar w_i=(i,i+1)$.
\end{definition}

\begin{lemma}\label{le:power}
Given $w\in W(\tilde A_{n-1})$, there is a positive integer $d$ such
that $w^d\in\bZ^{n-1}$.
\end{lemma}
\begin{proof}
Recall from~\eqref{eq:ZWS} that $W(\tilde A_{n-1})$ is a semidirect product
$\bZ^{n-1}\rtimes\sym{n}$. So if $d$ is the order of $\bar w$ then
$w^d\in\bZ^{n-1}$. 
\end{proof}

\begin{lemma}\label{le:ell4}
For elements of the subgroup $\bZ^{n-1}$, the vector $(r_1,\dots,r_n)$
is 
\[ [nr_1+1,nr_2+2\dots,nr_n+n] \]
in window notation. The length function is given by  
\begin{equation*}
 \ell(r_1,\dots,r_n)=\sum_{1\le i<j\le n}|r_j-r_i|. 
\end{equation*}
\end{lemma}
\begin{proof}
It follows from Lemma~\ref{le:ell3} that
\begin{align*}
\ell(r_1,\dots,r_n)&=\ell[nr_1+1,\dots,nr_n+n]\\
&=\sum_{1\le i<j\le n}
\left|\left\lfloor\frac{(nr_j+j)-(nr_i+i)}{n}\right\rfloor\right| \\
&=\sum_{1\le i<j\le n}
\left|\left\lfloor (r_j-r_i) + \frac{j-i}{n}\right\rfloor\right| \\
&=\sum_{1\le i<j\le n}|r_j-r_i|,
\end{align*}
since $r_j-r_i$ is an integer and $0< \frac{j-i}{n}<1$.
\end{proof}

\begin{theorem}\label{th:ell5}
Let 
\[ w=[s_1,\dots,s_n]\in  W(\tilde A_{n-1}) \] 
in window notation, and 
\[ w'=(r_1,\dots,r_n)=[nr_1+1,\dots,nr_n+n]\in \bZ^{n-1} \] 
in vector and window notation.
Then
\begin{align*}
ww'&=[nr_1+s_1,\dots,nr_n+s_n], \\
w'w&=[nr_{s_1}+s_1,\dots,nr_{s_n}+s_n]=w(w^{-1}w'w)
\end{align*}
where
$w^{-1}w'w=(r_{s_1},\dots,r_{s_n})=[nr_{s_1}+1,\dots,nr_{s_n}+n]$. In
the expression $r_{s_i}$, we read the subscript $s_i$ modulo $n$, as the
conjugation action of $w$ on $w'$ only depends on the image $\bar w$
of $w$ in $\sym{n}$. 
\begin{enumerate}
\item[(a)]
The following are equivalent.
\begin{enumerate}
\item[(1)] $\ell(ww')<\ell(w)+\ell(w')$,
\item[(2)] There is a pair of coordinates $1\le i<j\le n$ with either
  $r_i<r_j$ and $s_i>s_j$, or $r_i>r_j$ and $s_i+n<s_j$.
\end{enumerate}
\item[(a$'$)] The following are equivalent.
\begin{enumerate}
\item[(1)] $\ell(ww')=\ell(w)+\ell(w')$,
\item[(2)] For all $1\le i<j\le n$, either $r_i\ge r_j$ and $s_i+n>
  s_j$, or $r_i\le r_j$ and $s_i< s_j$.
\end{enumerate}
\item[(b)] The following are equivalent.
\begin{enumerate}
\item[(1)] $\ell(ww')=\ell(w')-\ell(w)$,
\item[(2)]
For all $1\le i<j\le n$ we have either 
\[
r_i-r_j\le \left\lfloor\frac{s_j-s_i}{n}\right\rfloor\le 
0,\quad\text{\rm or }\quad
r_i-r_j\ge \left\lfloor\frac{s_j-s_i}{n}\right\rfloor\ge 0.
\]
\end{enumerate}
\item [(c)] The following are equivalent.
\begin{enumerate}
  \item [(1)] $\ell(w'w)<\ell(w')+\ell(w)$,
  \item [(2)] There is a pair of coordinates $1\le i<j\le n$ with
    either $r_{s_i}<r_{s_j}$ and $s_i>s_j$, or $r_{s_i}>r_{s_j}$ and
    $s_i+n<s_j$. 
\end{enumerate}
\end{enumerate}
\end{theorem}
\begin{proof}
We have
\begin{align*}
 w&=[s_1,\dots,s_n] \colon jn+i \to jn+s_i\qquad (1\le i\le n), \\
 w'&= (r_1,\dots,r_n)=[nr_1+1,\dots,nr_n+n] \colon jn+i \to
     (j+r_i)n+i\qquad (1\le i\le n).
\end{align*}
Then the formulae for $ww'$ and $w'w$ come from composing these. Given
this, using Lemmas~\ref{le:ell3} and~\ref{le:ell4} we have
\begin{align*} 
\ell(ww')&=\sum_{1\le i<j\le n}
\left|\left\lfloor\frac{(nr_j+s_j)-(nr_i+s_i)}{n}\right\rfloor\right|
=\sum_{1\le i<j\le n}\left|(r_j-r_i)+\left\lfloor\frac{s_j-s_i}{n}\right\rfloor\right|.
\end{align*}
By the triangle inequality this is less than or equal to
\[ \ell(w)+\ell(w')=\sum_{1\le i<j\le n} |r_j-r_i|+\sum_{1\le i<j\le n}
\left|\left\lfloor\frac{s_j-s_i}{n}\right\rfloor\right|.
\]
The expression $s_j-s_i$ is never equal to $0$ or $n$; the expression
$\lfloor\frac{s_j-s_i}{n}\rfloor$ is therefore negative if
$s_j-s_i<0$, 
zero if $0<s_j-s_i<n$, and positive if $s_j-s_i>n$.
So we have equality
unless $r_i<r_j$ and $s_i>s_j$, or $r_i>r_j$ and $s_i+n<s_j$.
Thus in part (a), (1) and (2) are equivalent. Part (a$'$) is the
contrapositive of part (a).
For part (b), $\ell(ww')$ is equal to
\[ \ell(w')-\ell(w)=\sum_{1\le i<j\le n} |r_j-r_i|-\sum_{1\le i<j\le n}
\left|\left\lfloor\frac{s_j-s_i}{n}\right\rfloor\right|.
\]
if and only if each $-a\le b\le 0$ or $-a\ge b\ge 0$, where $a=r_j-r_i$
and $b=\lfloor \frac{s_j-s_i}{n}\rfloor$.

The proof of part (c) is similar to part (a$'$).
\end{proof}

\begin{corollary}\label{co:ellww'=w''w} 
Let $w=[s_1,\ldots,s_n]$, $w'=(r_1,\ldots,r_n)\in\bZ^{n-1}$ and
$w''=ww'w^{-1}=(t_1,\ldots,t_n)\in\bZ^{n-1}$ so that
$ww'=w''w$. Assume that $r_{i_1}\le\cdots\le r_{i_n}$ and
$s_{i_k}+n<s_{i_{k+1}}$ for all $1\le k\le n-1$. Then
$\ell(ww')=\ell(w)+\ell(w')$, $\ell(w''w)=\ell(w'')+\ell(w)$ and
$t_{s_k}=r_k$ for each $1\le k\le n$. 
\end{corollary}
\begin{proof} The fact $\ell(ww')=\ell(w)+\ell(w')$ follows from
  Theorem \ref{th:ell5}(a$'$). Let $w''=(t_1,\ldots,t_n)$ so that
  $r_j=t_{s_j}$. Let $i<j$, $i=i_a$ and $j=i_b$. If $a<b$, then
  $s_{i_a}+n<s_{i_b}$ but $t_{s_i}=r_{i_a}\le r_{i_b}= t_{s_j}$. If
  $a>b$, then $s_i>s_j$ but $t_{s_i}=r_{i_a}\ge r_{i_b}=t_{s_j}$. So
  $\ell(w''w)=\ell(w'')+\ell(w)$ by Theorem~\ref{th:ell5}(c). 
\end{proof}

\begin{corollary}\label{co:ell6}
For $w=(r_1,\dots,r_n)\in\bZ^{n-1}$ in vector notation, we have
\[ \ell(w_iw)<\ell(w) \Longleftrightarrow 
\begin{cases} 
r_i<r_{i+1}& 1\le i\le n-1,\\
r_n<r_1 & i=0
\end{cases} \]
while
\[ \ell(ww_i)<\ell(w)\Longleftrightarrow
\begin{cases}
r_i>r_{i+1}&1\le i\le n-1,\\
r_n>r_1&i=0.
\end{cases} \]
\end{corollary}
\begin{proof}
This follows immediately from Theorem~\ref{th:ell5}.
\end{proof}

\begin{corollary}\label{co:ellZ}
Let $w=(r_1,\dots,r_n)$ and $w'=(r'_1,\dots,r'_n)$ be elements of
$\bZ^{n-1}$. Then $ww'=w'w=(r_1+r'_1,\dots,r_n+r'_n)$, and the
following are equivalent:
\begin{enumerate}
\item $\ell(w)+\ell(w')=\ell(ww')$,
\item there do not exist coordinates $i$, $j$ such that $r_i<r_j$ but
  $r'_i>r'_j$, 
\item there exists an ordering on the coordinates for which both the
  $r_i$ and the $r'_i$ are non-decreasing.
\end{enumerate}
\end{corollary}
\begin{proof}
The equivalence of (1) and (2) follows from
Theorem~\ref{th:ell5}\,(a). The equivalence of (2) and (3) is straightforward.
\end{proof}

\begin{corollary}
The subalgebra $(k\bZ^{n-1})^{\sym{n}}$ is central in the group
algebra $kW(\tilde A_{n-1})$.\hfill\qed
\end{corollary}

\begin{corollary}\label{co:w'''}
Given $w\in W(\tilde A_{n-1})$, there exists $w'$ such that
$ww'$ is an element of $\bZ^{n-1}$ with no two coordinates equal and
$\ell(ww')=\ell(w)+\ell(w')$. Similarly, there exists 
$w''$ such that $w''w$ is an element of $\bZ^{n-1}$ with no  two
coordinates equal and $\ell(w''w)=\ell(w'')+\ell(w)$.
\end{corollary}
\begin{proof}
For this, we apply Theorem~\ref{th:ell5}\,(b). Write
$w^{-1}=[s_1,\dots,s_n]$. Let $\{i_1,\dots,i_n\}=\{1,\dots,n\}$ be
chosen so that $s_{i_1}\le \dots\le s_{i_n}$. Then
choose $(r_1,\dots,r_n)$ so that for $1\le j\le n-1$ we have
\[ r_{i_{j+1}}-r_{i_j} >
  \left\lfloor\frac{s_{i_{j+1}}-s_{i_j}}{n}\right\rfloor \]
and $\sum_ir_i=0$. This is clearly possible, by choosing the $r_{i_j}$
in ascending order and far enough apart, and then adjusting one end or
the other to make the sum equal to zero. Then condition (b)\,(2) holds.
Setting $w'''=(r_1,\dots,r_n)$, we have
$\ell(w^{-1}w''')=\ell(w''')-\ell(w^{-1})$. We now set $w'=w^{-1}w'''$, so
that we have $\ell(w')=\ell(w''')-\ell(w^{-1})$. 
We get $ww'=w'''\in\bZ^{n-1}$ and since $\ell(w^{-1})=\ell(w)$, we have
\[  \ell(ww')=\ell(w''')=\ell(w^{-1})+\ell(w')=\ell(w)+\ell(w'). \] 
The other case is similar.
\end{proof}

\begin{corollary}\label{co:ww'w}
Let $w\in W(\tilde A_{n-1})$ and $w'=(r'_1,\dots,r'_n)\in\bZ^{n-1}$ and
suppose that 
\[ \{1,\dots,n\}=\{i_1,\dots,i_n\}\quad \text{with} \quad
r'_{i_1}<\dots<r'_{i_n}. \] 
If $\ell(w'ww')=\ell(w)+2\ell(w')$
then $w=(r_1,\dots,r_n)\in\bZ^{n-1}$ with $r_{i_1}\le\cdots\le r_{i_n}$.
\end{corollary}
\begin{proof}
Let $w=[s_1,\ldots,s_n]$ and 
\[t:=[t_1,\ldots,t_n]=ww'=[s_1+nr'_1,\ldots,s_n+nr'_n].\] 
We claim that $t_{i_1}<\cdots<t_{i_n}$. Let $1\le a<b\le n$. We have
$r'_{i_a}<r'_{i_b}$. There are two cases, namely $i_a>i_b$ and
$i_a<i_b$. In the first case, since $\ell(ww')=\ell(w)+\ell(w')$, by
Theorem \ref{th:ell5}, we have $s_{i_b}+n\ge
s_{i_a}$. So 
\[t_{i_b}=s_{i_b}+nr'_{i_b}\ge s_{i_a}+n(r'_{i_b}-1)\ge
  s_{i_a}+nr'_{i_a}=t_{i_a}.\] 
In the second case, we have $s_{i_a}\le s_{i_b}$ and hence  
\begin{equation}\label{Eq:t}
t_{i_a}<t_{i_a}+n= s_{i_a}+nr'_{i_a}+n \le s_{i_b}+nr'_{i_b}=t_{i_b}.
\end{equation} 
We conclude that $t_{i_a}\le t_{i_b}$ in both cases. Since the
coordinates of $t$ must be all distinct, we have 
the desired claim. 

Next, we prove that $t_j\equiv j$ for each $1\le j\le n$. For each
$1\le k\le n$, let $t_{i_k}\equiv j_k$. It suffices to prove that
$j_k=i_k$. Let $1\le a<b\le n$. We have $t_{i_a}<t_{i_b}$. We now use
the other version of Theorem \ref{th:ell5} for
$\ell(w't)=\ell(w')+\ell(t)$. Again, we consider the two cases,
 $i_a>i_b$ and $i_a<i_b$. In the first case, we must have
$r'_{t_{i_b}}>r'_{t_{i_a}}$, i.e., $r'_{j_b}>r'_{j_a}$. In the second
case, we already have $t_{i_a}+n<t_{i_b}$ (see Equation \ref{Eq:t};
$t_{i_a}$ and $t_{i_b}$ cannot be congruent modulo $n$)
and hence we must have $r'_{t_{i_a}}<r'_{t_{i_b}}$, i.e.,
$r'_{j_a}<r'_{j_b}$. Therefore, in both cases, we have
$r'_{j_a}<r'_{j_b}$. This shows that $j_k=i_k$ for all $k$ and hence
$t_j\equiv j$ for all $1\le j\le n$.

Since $s_j\equiv t_j\equiv j$, we conclude that
$w=[s_1,\ldots,s_n]\in \bZ^{n-1}$. Let $w=(r_1,\ldots,r_n)$ where
$s_j=j+nr_j$. By Corollary \ref{co:ellZ}, we must have $r_{i_1}\le
\cdots\le r_{i_n}$.  
\end{proof}

\begin{lemma}\label{le:Aw''B} 
Let $\{i_1,\dots,i_n\}=\{j_1,\dots,j_n\}=\{1,\dots,n\}$, and 
let $\Omega$ and $\Omega'$ be subsets of $\bZ^{n-1}$ such that 
$r_{i_1}\le\cdots\le r_{i_n}$ and
$r'_{j_1}\le\cdots\le r'_{j_n}$ for all
$w=(r_1,\dots,r_n)\in\Omega$ 
and $w'=(r'_1,\dots,r'_n)\in\Omega'$.
Then there exists $w''=[c_1,\ldots,c_n]$ such that
\begin{enumerate}
\item
  $c_{j_k}+n<c_{j_{k+1}}$ for all $1\le k\le n-1$, 
\item
$c_{j_k}\equiv
  i_k\pmod n$ for all $1\le k\le n$, and
\item
  $\ell(ww''w')=\ell(w)+\ell(w'')+\ell(w')$ for all $w\in \Omega$ and
  all $w'\in \Omega'$. 
\end{enumerate}
Furthermore, for each $w'=(r_1',\ldots,r_n')\in\Omega'$, 
if we write $w''w'=w'''w''$ with $w'''=(t_1,\ldots,t_n)$, 
then $\ell(w''w')=\ell(w'')+\ell(w')$,
$\ell(w'''w'')=\ell(w''')+\ell(w'')$, $t_{c_k}=r'_k$ for
each $1\le k\le n$, and $t_{i_1}\le \cdots\le t_{i_n}$.
\end{lemma}
\begin{proof} 
Suppose that
$w=(r_1,\ldots,r_n)\in \Omega$ and $w'=(r'_1,\ldots,r'_n)\in \Omega'$. We define
$w''=[c_1,\ldots,c_n]$ in window notation, with 
\[ c_{j_u}=n(2u-n-1)+i_u. \]
Notice that these $c_{j_u}$ sum to $n(n+1)/2$ because the $i_u$ do,
and the additional terms add to zero. 
Let 
\[ s=[s_1,\ldots,s_n]=[c_1+nr'_1,\ldots,c_n+nr'_n]=w''w'. \]
If $1\le
u<v\le n$, then $r'_{j_u}\le r'_{j_v}$, and since $i_u$ and $i_v$ lie
between $1$ and $n$, and we have
\[ c_{j_u}=(2u-n-1)n+i_u<(2v-n-1)n+i_v=c_{j_v},\] 
and $s_{j_u}<s_{j_v}$.
By Theorem
\ref{th:ell5}\,(a$'$), it follows that $\ell(w''w')=\ell(w'')+\ell(w')$. 

We have $s_{j_u}=c_{j_u}+nr'_{j_u}\equiv c_{j_u}\equiv i_u \pmod{n}$,
and so the $j_u$th coordinate of $w(w''w')$ is 
\[ nr_{i_u}+s_{j_u}=nr_{i_u}+c_{j_u}+nr'_{j_u}. \]
If $1\le u< v\le n$ then $r_{i_u}\le r_{i_v}$ and $s_{j_u}<s_{j_v}$, and
so again using Theorem~\ref{th:ell5}\,(a$'$), we have
\[ \ell(w(w''w'))=\ell(w)+\ell(w''w')=\ell(w)+\ell(w'')+\ell(w'). \] 
Since $w''$ depends only on the $i_u$ and the $j_u$ (not on any
specific element 
in $\Omega$ nor $\Omega'$), we have the desired result for our first
assertion.  

We now prove the second assertion. This follows from
Corollary~\ref{co:ellww'=w''w}, except for 
  the inequality $t_{i_1}\le \cdots\le t_{i_n}$. However, since
  $c_{j_k}\equiv i_k$, we 
  have $t_{i_k}=t_{c_{j_k}}=r'_{j_k}$, and so $t_{i_k}=r'_{j_k}\le
  r'_{j_{k+1}}=t_{i_{k+1}}$. 
\end{proof}

\begin{theorem}\label{th:ww''w'}
Given $w$ and $w'\in W(\tilde A_{n-1})$, there exists $w''$ such that 
\[ \ell(ww''w')=\ell(w)+\ell(w'')+\ell(w'). \]
\end{theorem}
\begin{proof}
The special case when $w,w'\in\bZ^{n-1}$ has been dealt with in Lemma
\ref{le:Aw''B}. For the general case, given $w$ and $w'$, using
Corollary~\ref{co:w'''} we find elements $a$ and $b$ such that $wa$
and $bw'$ are in $\bZ^{n-1}$, $\ell(wa)=\ell(w)+\ell(a)$, and
$\ell(bw')=\ell(b)+\ell(w')$. Then using the special case proved
above, we may find $c$ such that 
\begin{align*}
\ell(w(acb)w')&=\ell((wa)c(bw'))\\
&=\ell(wa)+\ell(c)+\ell(bw')\\
&=\ell(w)+\ell(a)+\ell(c)+\ell(b)+\ell(w')\\
&\ge\ell(w)+\ell(acb)+\ell(w')\\
&\ge\ell(w(acb)w').
\end{align*}
It follows that both inequalities are equalities, and the theorem is
proved with $w''=acb$.
\end{proof}

\begin{proposition}\label{pr:w'w''}
Given $i$ with $0\le i\le n-1$ and $v=(r_1,\dots,r_n)\in\bZ^{n-1}$ in
vector notation, with $\bar w_i(v)\ne v$ (cf.\ 
Definition~\ref{def:wbar}), i.e.,
\[ \begin{cases} r_i\ne r_{i+1} & 1\le i\le n-1 \\ r_n\ne r_1 &
    i=0, \end{cases} \] 
the elements $w'_i=vw_i$,
$w''_i=w_iv$ in 
$W(\tilde A_{n-1})$ both have image $\bar w_i$ under the
map~\eqref{eq:ZWS}, and we have
$\ell(w'_iw''_i)=\ell(w'_i)+\ell(w''_i)$, and
$w'_iw''_i=v^2=(2r_1,\dots,2r_n)$.
\end{proposition}
\begin{proof}
By Theorem~\ref{th:ell5} we have
$\ell(w'_i)=\ell(v)\pm 1$, $\ell(w''_i)=\ell(v)\mp 1$. 
We also have $w'_iw''_i=vw_iw_iv=v^2$, so
$\ell(w'_iw''_i)=2\ell(v)=\ell(w'_i)+\ell(w''_i)$. 
\end{proof}

\section{The affine nilCoxeter algebra}\label{se:N}

Let $A=\mcN(\tilde A_{n-1})$ be the nilCoxeter algebra of type $\tilde
A_{n-1}$ over a field $k$. This is generated by elements $Y_0$, $Y_1$, \dots, 
$Y_{n-1}$, and with relations 
$Y_i^2=0$, $Y_iY_{i+1}Y_i=Y_{i+1}Y_iY_{i+1}$, and $Y_iY_j=Y_jY_i$, if 
$|i-j|\ge 2$,
where the indices in the relations are to be read modulo $n$. 
These relations are obtained from the relations for
$W(\tilde A_{n-1})$ by replacing $w$ with $Y$ throughout, except that 
the relation $w_i^2=1$ is replaced with $Y_i^2=0$. 
If $w=w_{i_1}\dots w_{i_m}$ is a reduced word in $W(\tilde A_{n-1})$,
we write $Y_w$ for the corresponding word $Y_{i_1}\dots Y_{i_m}$ in
$A$. Reduced words for the same $w$ give equal $Y_w$, so this notation
is not ambiguous. These reduced words $Y_w$ form a basis for $A$. Thus $A$ is
the semigroup algebra of the semigroup generated by 
$Y_0,\dots,Y_{n-1}$ and an absorbing element $0$, with these 
relations. 
The length of $Y_w$ is defined to be $\ell(Y_w)=\ell(w)$;
we do not assign a length to zero,
nor to elements that are not monomials, but we do assign length zero
to the identity element. The
relations are all homogeneous with respect to length, so it is
possible to use length to regard $A$ as an algebra graded by the
non-negative integers. The multiplication of words $Y_w$ and $Y_{w'}$
is given by
\begin{equation}\label{eq:YwYw'}
Y_wY_{w'}=\begin{cases} Y_{ww'} & \ell(ww')=\ell(w)+\ell(w') \\ 0 &
    \ell(ww') < \ell(w)+\ell(w'). \end{cases} 
\end{equation}

\begin{remark}\label{rk:anti-aut}
There is an anti-automorphism of $A$ that fixes each $Y_i$ and
reverses every word in the $Y_i$. Thus we have $A\cong A^\op$.
\end{remark}

One obvious consequence of the definition is the following.

\begin{lemma}\label{le:cancel}
Let $w$, $w'$ and $w''$ be elements of $W(\tilde A_{n-1})$. In the
algebra $A$, if
$Y_wY_{w''}=Y_{w'}Y_{w''}\ne 0$, or if
$Y_{w''}Y_w=Y_{w''}Y_{w'}\ne 0$,  then $w=w'$ and $Y_w=Y_{w'}$. 
\end{lemma}
\begin{proof}
If $Y_wY_{w''}\ne 0$ then by \eqref{eq:YwYw'} it is equal to $Y_{ww''}$. So if
$Y_wY_{w''}=Y_{w'}Y_{w''}\ne 0$ then $ww''=w'w''$ in 
$W(\tilde A_{n-1})$, and so $w=w'$. Similarly for the other case.
\end{proof}

\begin{proposition}\label{pr:genfn}
Let $J$ be the ideal in $A$ generated by
$Y_0,\dots,Y_{n-1}$. Then the generating function for the dimensions
of the quotients $J^i/J^{i+1}$ (where $J^0=A$) is given by
\begin{equation*}
\sum_{i=0}^\infty t^i\dim_k J^i/J^{i+1}=\frac{1-t^n}{(1-t)^n}.
\end{equation*}
\end{proposition}
\begin{proof}
This follows from Corollary~\ref{co:Poincareseries}.
\end{proof}

\begin{remark}
The expression in Proposition~\ref{pr:genfn} is
the generating function for the number of distinct non-zero
monomial elements of $A$, counted by length. 
\end{remark}

\begin{lemma}\label{le:prime}
Given words $w$ and $w'$ there exists a word $w''$ such that
$Y_wY_{w''}Y_{w'}\ne 0$.
\end{lemma}
\begin{proof}
This follows from Theorem~\ref{th:ww''w'}.
\end{proof}

\begin{notation}\label{not:Y}
We shall sometimes write $Y_{i_1\dots i_m}$ for $Y_{i_1}\dots Y_{i_m}$
for convenience, especially in the examples at the end of the
paper. So for example $Y_1Y_2Y_1$ is abbreviated to $Y_{121}$. 
\end{notation}

\section{A large commutative subalgebra}\label{se:C}

The algebra $A$ has a large commutative subalgebra $C$.
This is the subalgebra
spanned by the words corresponding to the words in $W(\tilde A_{n-1})$
for elements of the normal subgroup $\bZ^{n-1}$. 
Then $C$ has a basis corresponding to the integer vectors
$(r_1,\dots,r_n)$ with $\sum_i r_i=0$, and we refer to this as 
\emph{vector notation} for elements of $C$. If $v$ is such a vector, let us
write $Z_v$ for the corresponding basis element of $C$. 

\begin{lemma}\label{le:ZvZv'}
If $v=(r_1,\dots,r_n)$ is a vector in $\bZ^{n-1}$, define the
\emph{length} of $v$ to be 
\[ \ell(v)=\sum_{i<j}|r_i-r_j|, \]
and the length of $Z_v$ to be $\ell(Z_v)$ as defined at the beginning
of Section~\ref{se:N}.
Then the length function on basis elements of $C$ is given by
$\ell(Z_v)=\ell(v)$.

The multiplication on $C$ is given as follows.
Let $v=(r_1,\dots,r_n)$, $v'=(r'_1,\dots,r'_n)\in\bZ^{n-1}$. If there
are coordinates $i$, $j$ with $r_i<r_j$ and $r'_i>r'_j$ then
$Z_vZ_{v'}=0$. Otherwise we have
$Z_vZ_{v'}=Z_{v+v'}$.
\end{lemma}
\begin{proof}
The first statement follows from Lemma~\ref{le:ell4}.
The second statement then follows from Corollary~\ref{co:ellZ}.
\end{proof}

The algebra $C$ is $\bZ^{n-1}$-graded, with each graded piece one
dimensional, spanned by the corresponding $Z_v$. The group
$\sym{n}$ acts as algebra automorphisms on $C$ by permuting the $n$
coordinates in vector notation: for $\sigma\in \sym{n}$ and
$v=(r_1,\dots,r_n)$ we have $\sigma(v)=(r_{\sigma^{-1}(1)},\dots,r_{\sigma^{-1}(n)})$.

\begin{lemma}\label{le:YiZv}
Let $v=(r_1,\dots,r_n)\in\bZ^{n-1}$
and let $Z_v$ be the corresponding basis
vector in $C$. Then for $0\le i\le n-1$ we have 
$Y_iZ_v=Z_{\bar w_i(v)}Y_i$ (cf.~Definition~\ref{def:wbar}), and
\[ Y_iZ_v=0 \Longleftrightarrow\begin{cases}
r_i<r_{i+1}& 1\le i\le n-1\\ r_n<r_1 &i=0\end{cases}, \] 
while
\[ Z_vY_i=0\Longleftrightarrow\begin{cases}
r_i>r_{i+1}&1\le i\le n-1\\r_n>r_1&i=0.\end{cases} \]
\end{lemma}
\begin{proof}
This follows from Theorem~\ref{th:ell5}, Corollary~\ref{co:ell6},
and Equation~\eqref{eq:YwYw'}.
\end{proof}

\begin{lemma}\label{le:Ywd}
Given $w\in W(\tilde A_{n-1})$, there is a positive integer $d$ such
that $Y_w^d\in C$.
\end{lemma}
\begin{proof}
By Lemma~\ref{le:power}, there is a positive integer $d$ such that
$w^d\in\bZ^{n-1}$. Using~\eqref{eq:YwYw'}, 
if $\ell(w^d)=d\ell(w)$, then $Y_w^d=Y_{w^d}\in
C$, and if not then we have $Y_w^d=0\in C$.
\end{proof}

\begin{lemma}\label{le:Cfg/k}
The algebra $C$ is finitely generated over $k$, and hence Noetherian
as a commutative ring.
\end{lemma}
\begin{proof}
We may either refer to Gordan~\cite{Gordan:1873a}, or
Dickson~\cite{Dickson:1913a}, or argue directly as follows. Suppose
without loss of generality that $v=(r_1,\dots,r_n)$ with 
$r_1\le r_2\le \dots \le r_n$. Consider
the vectors $v_1=(1-n,1,\dots,1)$, $v_2=(2-n,2-n,2,2,\dots,2)$, \dots,
$v_{n-1}=(-1,\dots,-1,n-1)$. If any $r_{i+1}-r_i$ is greater than $n$
then $\ell(v)=\ell(v-v_i)+\ell(v_i)$ and so by
Corollary~\ref{co:ellZ}, $Z_v=Z_{v-v_i}Z_{v_i}$. This shows that all
$Z_v$ with $v$ of
this form are generated by those with $r_{i+1}-r_i$ at most
$n$. There are only finitely many of these, since $\sum_i r_i=0$.
\end{proof}

\begin{proposition}\label{pr:minimal-primes}
Every associated prime in $C$ is minimal.
There are exactly $n!$ minimal primes in $C$, corresponding to the
total orders on the coordinates. The prime $\mfp$ corresponding to an
order is generated by all the $Z_v$ such that the coordinates of $v$ are
not ordered this way, and the quotient $C/\mfp$ has a basis consisting
of the $Z_v$ that are non-descending with respect to this order. The
action of $\sym{n}$ on the minimal primes is transitive, with trivial stabiliser.
The algebra $C$ is a reduced ring.
\end{proposition}
\begin{proof}
An associated prime of $C$ is a prime ideal that is the annihilator of some
element, see for example Matsumura~\cite[\S6]{Matsumura:1986a}
(applied to the $C$-module $C$). The
minimal elements of the set of associated primes are exactly the
minimal primes~\cite[Theorem~6.5\,(iii)]{Matsumura:1986a}.

Since $C$ is graded by length, and is the semigroup ring of the
semigroup of monomials in it, the annihilator of a sum of monomials is
the intersection of the annihilators. If a finite intersection of ideals is
prime, then it is one of those ideals, so we can restrict our
attention to annihilators of monomials. If $v=(r_1,\dots,r_n)$, 
by Lemma~\ref{le:ZvZv'}, the
annihilator of $Z_v$ is spanned by the $Z_{v'}$ where
$v'=(r'_1,\dots,r'_n)$ and there exist coordinates $i$ and $j$ with
$r_i<r_j$ and $r'_i>r'_j$. If $v$ has two equal entries $r_i=r_j$
then there are two elements $Z_{v'}$ and
$Z_{v''}$ that do not annihilate $Z_v$, but such that $Z_{v'}Z_{v''}=0$, and
so the annihilator is not prime. On the other hand, if all the
coordinates of $v$ are distinct, then the coordinates define a total
order of the coordinates, and the annihilator is spanned by
those $Z_{v'}$ whose coordinates are not compatible with this total
order. The quotient by this annihilator has a basis consisting of the
$Z_{v'}$ where the coordinates are compatible with this total order.
These satisfy $Z_{v'}Z_{v''}=Z_{v'+v''}$, and so this quotient is an
integral domain. 

It follows that there are $n!$ prime annihilators in $C$,
corresponding to the $n!$ total orders on the coordinates. None of
these contains any other, so they are all minimal primes.
Since the intersection of the minimal primes is zero, $C$ is a reduced ring.
\end{proof}

\begin{definition}\label{def:p0}
We write $\mfp_0$ for the minimal prime ideal in $C$ corresponding to the order 
\[ 1<2<\dots<n. \] 
as in Proposition~\ref{pr:minimal-primes}. Thus $\mfp_0$ is spanned by
those $Z_v$ where $v=(r_1,\dots,r_n)$, such that $v$ does not
satisfy $r_1\le r_2\le \dots \le r_n$.
\end{definition}

The action of 
$\sym{n}$ on $C$ by permuting the $n$ coordinates in vector
notation induces 
an action on the set of ideals of $C$. By the theorem,
it acts regularly on the minimal primes. So the orbit of any
prime ideal has a representative that contains $\mfp_0$.

\begin{theorem}\label{th:fg}
The algebra $A$ is finitely generated as a right module and as a left
module over the commutative 
subalgebra $C$.
\end{theorem}
\begin{proof}
We shall prove that the algebra $A$ is finitely generated as a right
$C$-module. Since the anti-auto\-morph\-ism of $A$ described in 
Remark~\ref{rk:anti-aut} restricts to an automorphism of the
commutative subalgebra $C$,  it will follow that $A$ is
also finitely generated as a left $C$-module.

As in the proof of Lemma~\ref{le:Cfg/k}, we consider the vectors
\[ v_1=(1-n,1,\dots,1),\ v_2=(2-n,2-n,2,2,\dots,2),\ \dots,\
v_{n-1}=(-1,\dots,-1,n-1). \]

Consider a word $w$ in window notation $[s_1,\dots,s_n]$. Let
$\sigma\in \sym{n}$ satisfy $s_{\sigma(1)}\le s_{\sigma(2)}\le
\cdots \le s_{\sigma(n)}$, and let $v'_1,\dots,v'_{n-1}$ be defined by
applying $\sigma^{-1}$ to the coordinates of $v_1,\dots,v_{n-1}$ in
vector notation.
Then the $j$th entry of $v'_i$ in window notation is either $-n^2+ni+j$
or $ni+j$.

If some $s_{\sigma(i+1)}-s_{\sigma(i)}>n^2$, then let $v=v_i'$ and
  $w'$ be the word whose window notation is $[s'_1,\ldots,s'_n]$ where
  $s'_j$ is equal to $s_j+n^2-in$ if $\sigma^{-1}(j)\le i$ and
  $s_j-ni$ if $\sigma^{-1}(j)>i$. We have $w'v_i'=w$ and claim that
  $\ell(w'v)=\ell(w')+\ell(v)$. Notice
  that \[v=((v_i)_{\sigma^{-1}(1)},\ldots,
    (v_i)_{\sigma^{-1}(n)})=[1+n(v_i)_{\sigma^{-1}(1)},\ldots,
    n+n(v_i)_{\sigma^{-1}(n)}].\] Suppose on the contrary that
  $\ell(w'v)<\ell(w')+\ell(v)$. By Theorem~\ref{th:ell5}, we consider
  the case (the other case is similar) where $1\le j<k\le n$ with
  $v_j<v_k$ and $s'_j>s'_k$. The only possibility for  $v_j<v_k$ is
  when $v_j=i-n$ and $v_k=i$. So $\sigma^{-1}(j)\le i$ and
  $\sigma^{-1}(k)> i$. So we have $s_j+n^2>s_k$. But, by the choice of
  $\sigma$, we also have \[s_j\le s_{\sigma(i)}\le
    s_{\sigma(i+1)}-n^2\le s_k-n^2.\] A contradiction. It follows that
  $Y_{w'}Z_v=Y_w$.

It follows that $A$ is generated as a right module over $C$ by elements
$Y_w$ where the window notation for $w$ is $[s_1,\dots,s_n]$, and
each $s_{\sigma(i+1)}-s_{\sigma(i)}\le n^2$, where $\sigma$ is defined
as above in terms of $s_1,\dots,s_n$, and $i$ runs from $1$ to $n-1$.
 Since $\sum_{i=1}^n s_i=n(n+1)/2$, there are only a finite number of
such $w$.
\end{proof}

\begin{corollary}\label{co:PI}
$A$ is an affine PI algebra.
\end{corollary}
\begin{proof}
This follows from Lemma~\ref{le:Cfg/k} and 
Theorem~\ref{th:fg} together with McConnell and
Robson \cite[Corollary~13.1.13]{McConnell/Robson:2001a}. 
\end{proof}

\begin{theorem}\label{th:fd}
Every simple $A$-module is finite dimensional.
\end{theorem}
\begin{proof}
This follows from Corollary~\ref{co:PI}
and~\cite[Theorem~13.10.3]{McConnell/Robson:2001a}. 
\end{proof}

An alternative proof will be given in Section~\ref{se:Amodules}

\begin{lemma}\label{le:primeC}
Let $x=\sum_w\lambda_w Y_w$ is a non-zero element of $A$. Then there
exist $w'$ and $w'' \in W(\tilde A_{n-1})$ and a minimal prime $\mfp$
of $C$ such that $Y_{w'} x Y_{w''}$ is an element of $C$ with non-zero
image in $C/\mfp$.
\end{lemma}
\begin{proof}
Choose $w$ with $\lambda_w\ne 0$. Multiplying on the right by a
suitable basis element, Corollary~\ref{co:w'''} shows that without loss of
generality we may assume that $w\in\bZ^{n-1}$ with no two coordinates
equal. 
If $z$ is another element with $\lambda_z\ne 0$
then by Corollary~\ref{co:ww'w}, either $Y_wY_zY_w=0$ or $Y_z\in C$ with
non-zero image in $C/\mfp$. Notice that $Y_w^3\neq 0$ by Corollary
\ref{co:ellZ}. Therefore, $Y_wxY_w$ has nonzero image in $C/\mfp$
using Proposition \ref{pr:minimal-primes}. 
\end{proof}

\begin{theorem}\label{th:prime}
The algebra $A$ is prime.
\end{theorem}
\begin{proof}
By definition, we must prove that given non-zero elements $x$ and $y$
of $A$, there exists $z$ such that $xzy\ne 0$. It's clearly enough to
show that there are elements $z_{1},\dots,z_{5}$ such that
$(z_{1}xz_{2})z_{3}(z_{4}yz_{5})\ne 0$.
Lemma~\ref{le:primeC} provides elements $z_1$, $z_2$, $z_4$, $z_5$
and minimal primes $\mfp$ and $\mfp'$ of $C$ 
such that $z_1xz_2$ and $z_4yz_5$  are elements of $C$ with
$z_1xz_2$ non-zero modulo $\mfp$ and $z_4yz_5$ non-zero modulo $\mfp'$.
Replacing $x$ by $z_1xz_2$ and $y$ by $z_4yz_5$, we may assume that
$x$ and $y$ are in $C$ with $x\not\in\mfp$ and $y\not\in\mfp'$.

Let $\Omega$ be the (non-empty) set consisting of $Z_v\in C$ with
$Z_v\not\in \mfp$, such that $Z_v$ appears with nonzero coefficient in $x$. Similarly,
we define the set $\Omega'$ for the prime $\mfp'$ and $y$. By
Proposition~\ref{pr:minimal-primes}, let $\{i_1,\ldots,i_n\}$ and
$\{j_1,\ldots,j_n\}$ be the total orders with respect to $\mfp$ and
$\mfp'$ respectively. Therefore, there exists $w=[c_1,\ldots,c_n]$
satisfying the conclusion of Lemma~\ref{le:Aw''B}. Let $z=Y_{w}$. Then
$Z_vzZ_{v'}\ne 0$ for each $Z_v\in \Omega$ and $Z_{v'}\in\Omega'$. For
$Z_{v'}\in\mfp'$, we claim that $zZ_{v'}=0$. By Proposition
\ref{pr:minimal-primes}, there exist $1\le a<b\le n$ such that
$t_{j_a}>t_{j_b}$ where $v'=(t_1,\ldots,t_n)$. The pair $j_a,j_b$
meets condition (2) in Theorem~\ref{th:ell5}(a). Similarly, $Z_vz=0$
for $Z_v\in\mfp$. Therefore, we may further assume that $\Omega$
(respectively, $\Omega'$) gives us all non-zero summands in $x$
(respectively, $y$). For each $Z_v\in \Omega$ and $Z_{v'}\in\Omega'$,
we use the second assertion in Lemma~\ref{le:Aw''B} to rewrite
$Z_vzZ_{v'}=Z_vZ_{\tilde v'}z$ with $Z_{\tilde v'}\not\in \mfp$
because $\tilde v'$ satisfies the total order
$i_1<\cdots<i_n$. Clearly, $Z_{\tilde v'}$'s are distinct for distinct
$Z_{v'}$'s. Therefore, $xzy=x\tilde yz$ where $\tilde
y\not\in\mfp$. Since $\mfp$ is prime and both $x,\tilde y$ are not in
$\mfp$, we must have $x\tilde y\not\in \mfp$. Lastly, $x\tilde yz\ne
0$ because, for each $Z_v\in\Omega$ and $Z_{v'}\in\Omega'$, we
have  
\begin{equation*}
\ell(vwv')=\ell(v)+\ell(w)+\ell(v')=\ell(v)+
\ell(\tilde v')+\ell(w)=\ell(v\tilde v'w).
\qedhere
\end{equation*} 
\end{proof}

\section{\texorpdfstring{The centre of $A$}{The centre of A}}\label{se:centre}

Recall that $A$ is an affine nilCoxeter algebra of type $\tilde
A_{n-1}$, and $C$ is the commutative subalgebra generated by elements
$Z_v$. The prime ideal $\mfp_0$ is spanned by those $Z_v$ with
$v=(r_1,\dots,r_n)$ that do not satisfy $r_1\le \dots\le r_n$, and the
integral domain $C/\mfp_0$ has a basis consisting of the images of those $Z_v$ that do
satisfy this condition.

In this section, we examine the invariants of the $\sym{n}$ action on
$C$. We show that $R=C^{\sym{n}}$ is an integral domain isomorphic to
  $C/\mfp_0$, and is equal to the centre in $A$.

\begin{definition}
If $v=(r_1,\dots,r_n)$ with $r_1\le \dots \le r_n$ and $\sum_{i=1}^nr_i=0$, we define 
$\hat Z_v$ to be the sum of the distinct images of $Z_v$ under the
action of $\sym{n}$.
\end{definition}

See for example Section~\ref{se:A2tilde} for the elements $Z_v$ and
$\tilde Z_v$ in the case $n=3$.

\begin{theorem}\label{th:Cfg/Csigma}
The algebra $C$ is finitely generated as a module for $R=C^{\sym{n}}$,
and $R$ is a finitely generated algebra.
\end{theorem}
\begin{proof}
In general, if a finite group acts as algebra automorphisms on a
finitely generated commutative 
$k$-algebra, then the algebra is finitely generated as a module over
the invariants, and the invariants are also a
finitely generated commutative algebra.
See for example 
Bourbaki~\cite[Ch.~V, \S1, no.~9, Theorem~2]{Bourbaki:1985a}.
The steps in the proof are as follows. First show that $C$ is integral over
$R$, because each $c\in C$ satisfies
$\prod_{\sigma\in\sym{n}}(X-\sigma(c))=0$, and the coefficients are in
$R$. 
Then show using induction on a finite set of algebra
generators, that $C$ is finitely generated as a module over
$R$. Finally use the Artin--Tate
lemma~\cite[Theorem~1]{Artin/Tate:1951a} to prove that 
$R$ is finitely generated as an algebra.
\end{proof}

\begin{theorem}\label{th:Zvhat}
If $v=(r_1,\dots,r_n)$ with $r_1\le \cdots \le r_n$ and
$\sum_{i=1}^nr_i=0$, and
$v'=(r'_1,\dots,r'_n)$ with $r'_1\le \cdots \le r'_n$ and $\sum_{i=1}^nr'_i=0$ then we
have $\hat Z_v\hat Z_{v'}=\hat Z_{v+v'}$.
These elements $\hat Z_v$ span
$R$. The inclusion of $R$ in $C$ induces an
isomorphism between $R$ and $C/\mfp_0$ sending $\hat Z_v$ to $Z_v$.
\end{theorem}
\begin{proof}
Let $\sigma,\sigma'\in\sym{n}$, and consider the product
$Z_{\sigma(v)}Z_{\sigma'(v')}$. For this to be non-zero, by Lemma~\ref{le:ZvZv'}
there is an order for the coordinates for which the coordinates of
$\sigma(v)$ and $\sigma'(v')$ are both in order. So there is a
permutation $\sigma''\in\sym{n}$ such that $\sigma(v)=\sigma''(v)$ and
$\sigma'(v')=\sigma''(v')$. Thus
$\sigma(v)+\sigma'(v')=\sigma''(v+v')$, and
$Z_{\sigma(v)}Z_{\sigma'(v')}=Z_{\sigma''(v+v')}$. Furthermore, no
other pairs of images of $v$ and $v'$ will add to give
$\sigma''(v+v')$. Thus we have $\hat Z_v\hat Z_{v'}=\hat Z_{v+v'}$.
The composite $R\to C \to C/\mfp_0$ sends $\hat Z_v$ to $Z_v$,
and is therefore an isomorphism.
\end{proof}

\begin{corollary}\label{co:domain}
The algebra $R$ is an integral domain.
\end{corollary}
\begin{proof}
Since $R$ is isomorphic to $C/\mfp_0$, and $\mfp_0$ is prime,  $R$  is an integral
domain. 
\end{proof}

\begin{lemma}\label{le:stab(p)}
Let $\mfp$ be a prime ideal of $C$ containing $\mfp_0$. Then the stabiliser of $\mfp$ in
the action of $\sym{n}$ on $C$ is a subgroup $H\le \sym{n}$ of the form
$H=\sym{n_1}\times\dots\times\sym{n_s}$, where each
$\sym{n_i}$ is acting on $n_i$ consecutive numbers between $1$ and
$n$. The group $H$ acts trivially on $C/\mfp$.
\end{lemma}
\begin{proof}
Recall from Definition~\ref{def:p0} that $v=(r_1,\dots,r_n)\in\bZ^{n-1}$ satisfies
$Z_v\not\in\mfp_0$ if and only if $r_1\le\dots\le r_n$. If
$\sigma\in\sym{n}$ stabilises $\mfp$ and $Z_v\not\in\mfp$ then
$Z_v\not\in\sigma^{-1}\mfp$ and so
$\sigma(Z_v)=Z_{\sigma(v)}\not\in\mfp$. It follows that
$r_{\sigma(1)}\le\dots\le r_{\sigma(n)}$. Thus the stabiliser of
$\mfp$ is contained in
\[ H=\{\sigma\in \sym{n}\mid\forall v\in\bZ^{n-1},\ Z_v\not\in\mfp\Rightarrow
  r_{\sigma(1)}\le\dots\le r_{\sigma(n)}\}. \]
This subgroup $H$ has the form $\sym{n_1}\times\dots\times\sym{n_s}$,
where each $\sym{n_i}$ is acting on $n_i$ consecutive numbers between
$1$ and $n$.

Conversely, if $\sigma\in H$ then $\sigma$ fixes each $v\in\bZ^{n-1}$
with $Z_v\not\in\mfp$, and so it fixes which linear combinations of
these $Z_v$ land
in $\mfp$. So $H$ stabilises $\mfp$ setwise and $C/\mfp$ pointwise.
\end{proof}

\begin{theorem}\label{th:SpecC}
The spectrum $\Spec(C)$ is a union of $n!$ irreducible components, all
isomorphic to $\Spec(R)$, and regularly permuted by
$\sym{n}$. The intersections of distinct irreducible components are
equal to
the fixed points of Young subgroups of $\sym{n}$. The primes in $\Spec(C)$ lying
above a given prime in $\Spec(R)$ are transitively permuted
by $\sym{n}$, with stabiliser a Young subgroup. Thus
\[ \Spec(R)\cong\Spec(C)/\sym{n}. \]
\end{theorem}
\begin{proof}
This follows from Theorem~\ref{th:Zvhat}  and Lemma~\ref{le:stab(p)}.
\end{proof}

The following theorem is proved as Proposition~9.1
in~\cite{Lam:2006a}. For the convenience of the reader, we rewrite the
proof given there in our notation.

\begin{theorem}\label{th:centre}
The algebra $R=C^{\sym{n}}$ is equal to the centre of $A$.
\end{theorem}
\begin{proof}
It follows from Lemma
\ref{le:YiZv} that for all $v\in \bZ^{n-1}$ and all $0 \le i \le n -1$
we have $Y_i \hat Z_v = \hat Z_vY_i$. Since the $Y_i$ generate $A$,
$R$ belongs in the centre of $A$. 

Conversely, we let $z\in Z(A)$, and we need to show that $z\in
R$. We may assume that $z$ is homogeneous. Suppose that the
coefficient of $Y_{v\sigma}$ in $z$ is non-zero, where $v\sigma\in \bZ^{n-1}\rtimes
\sym{n}$. We claim that $\sigma=1$. Let $v=(v_1,\ldots,v_n)$ so that 
\[v\sigma=[\sigma(1)+nv_{\sigma(1)},\ldots, \sigma(n)+nv_{\sigma(n)}].\] 
Let $i_1,\ldots,i_n$ such that
$\sigma(i_j)+nv_{\sigma(i_j)}<\sigma(i_{j+1})+nv_{\sigma(i_{j+1})}$
for each $1\le j\le n-1$. Given $r=(r_1,\ldots,r_n)\in \bZ^{n-1}$
such that $r_{\sigma(i_1)}<r_{\sigma(i_2)}<\cdots<r_{\sigma(i_n)}$, 
we claim that $\ell(Z_rY_{v\sigma})=\ell(Z_r)+\ell(Y_{v\sigma})$.
Let $1\le i<j\le n$ and
$\sigma(i)+nv_{\sigma(i)}>\sigma(j)+nv_{\sigma(j)}$. Then $i=i_b$ and
$j=i_a$ for some $1\le a<b\le n$. So
$r_{\sigma(j)}<r_{\sigma(i)}$. Similarly, let $1\le i<j\le n$ and
$\sigma(i)+nv_{\sigma(i)}+n<\sigma(j)+nv_{\sigma(j)}$. Then $i=i_a$
and $j=i_b$ for some $1\le a<b\le n$. So
$r_{\sigma(i)}<r_{\sigma(j)}$. By Theorem~\ref{th:ell5}\,(c), we have
the desired claim.
Suppose to the contrary that
$\sigma\neq 1$ and let $2\le m\le n$ be the largest $m$ such that 
$q:=\sigma(i_m)\neq i_m$. Since $z$ is a finite sum, let $N$ be greater
than $v'_{q}-v_{q}$ where $v'\in\bZ^{n-1}$ runs over all
$Y_{v'\sigma}$ which appears in $z$. Notice that $\sigma^{-1}(i_m)=i_c$ for some $c<m$. We pick $r$ so that the $r_i$'s satisfy the order we mentioned above and $r_{\sigma(i_m)}-r_{\sigma(i_c)}>N$. Since $z\in Z(A)$ and $Z_rY_{v\sigma}\neq 0$, we have
$Z_rY_{v\sigma}=Y_{v'\sigma}Z_r$ for some $Y_{v'\sigma}$ which appears
in $z$. Since $rv\sigma=v'\sigma r$, we
have \[(r_1+v_1,\ldots,r_n+v_n)=(v'_1+r_{\sigma^{-1}(1)},\ldots,
  v'_n+r_{\sigma^{-1}(n)}).\] In particular,
$r_q+v_q=v'_q+r_{\sigma^{-1}(q)}$ and hence  
\[r_{\sigma(i_m)}-r_{\sigma(i_c)}=r_q-r_{\sigma^{-1}(q)}= v'_q-v_q<N.\] A contradiction. Therefore, there is no such $m$, i.e., $\sigma=1$. 
Suppose now that $v_i\neq v_{i+1}$ and $v'$ be obtained from $v$ 
by swapping $v_i$ and $v_{i+1}$. By Corollary \ref{co:ell6}, \begin{align*}
Y_iZ_v&=\left \{\begin{array}{ll}Y_{w_iv}&
\text{if $v_i>v_{i+1}$,}\\ 0&\text{otherwise,}\end{array}\right .\\
Z_{v'}Y_i&=\left \{\begin{array}{ll}Y_{v'w_i}
&\text{if $v_i>v_{i+1}$,}\\ 0&\text{otherwise,}\end{array}\right .
\end{align*} Since $v'w_i=w_iv$, we conclude that the coefficients of
$Z_v$ and $Z_{v'}$ in $a$ are the same. Therefore, $z\in
R$. 
\end{proof}

\begin{lemma}
The algebra $A$ is finitely generated as a module over $R$.
\end{lemma}
\begin{proof}
By Theorem~\ref{th:Cfg/Csigma} the algebra $C$ is finitely generated
as a module over $R$, 
and by Theorem~\ref{th:fg}, $A$ is finitely generated as a module over $C$.
\end{proof}

\section{Ring theory for the centre}\label{se:R}

Recall from Theorem~\ref{th:centre} that $R=C^{\sym{n}}$ is the centre
of $A$, and recall from Theorem~\ref{th:Zvhat} that the inclusion of
$R$ into $C$ induces an isomorphism $R\cong C/\mfp_0$ sending $\hat
Z_v$ to $Z_v$, so that $R$ is an integral domain. The structure of $R$
is elucidated by the following theorem.

\begin{theorem}\label{th:R}
The algebra $R$ is isomorphic to the subalgebra of
$k[X_1,\dots,X_{n-1}]$ spanned by those monomials $X_1^{m_1}\dots
X_{n-1}^{m_{n-1}}$ where $m_1+2m_2+\dots+(n-1)m_{n-1}$ is a multiple
of $n$. The isomorphism takes $\hat Z_v$ to
$X_1^{r_2-r_1}X_2^{r_3-r_2}\dots X_{n-1}^{r_n-r_{n-1}}$, where $v=(r_1,\dots,r_n)$.
\end{theorem}
\begin{proof}
It is easy to check that this correspondence of basis elements
preserves multiplication. Note that 
\[ (r_2-r_1)+2(r_3-r_2)+\dots+(n-1)(r_n-r_{n-1})=nr_n-\sum_{i=1}^n r_i
  = nr_n \]
is a multiple of $n$. So the differences $r_{i+1}-r_i$ determine
$nr_n$, and hence determine $v$.
\end{proof}

\begin{corollary}\label{co:Noether-norm}
A Noether normalisation for $R$ is given by the elements $\hat
Z_{v_1},\dots,\hat Z_{v_{n-1}}$ where
$v_i=(i-n,\dots,i-n,i,\dots,i)$ with $i$ copies of $i-n$ and $n-i$
copies of $i$.
For every basis element $\hat Z_v$ of $R$, $\hat Z_v^n$ is a monomial
in the elements $\hat Z_{v_1}\dots,\hat Z_{v_{n-1}}$.
\end{corollary}
\begin{proof}
These elements correspond to the elements $X_1^n,\dots,X_{n-1}^n$
under the correspondence given in Theorem~\ref{th:R}.
\end{proof}

\begin{remark}
Let $v_i=(i-n,\dots,i-n,i,\dots,i)$, as in Corollary~\ref{co:Noether-norm}.
If $i$ and $n-i$ have a common factor $d$ then a smaller normalisation
can be obtained by replacing $v$ by $v/d$, with $i$ copies of
$(i-n)/d$ and $n-i$ copies of $i/d$.
The elements $Z_v\in C$ whose orbit sums are $\hat Z_v$  are
themselves powers, but of
elements not necessarily in $C$:
$Z_v=(Y_{i+1\,i+2\,\dots\,n-1,\,i\ i-1\,\dots\,1\,0})^{i(n-i)}$,
and similarly we have
$Z_{v/d}=(Y_{i+1\,i+2\,\dots\,n-1,\,i\ i-1\,\dots\,1\,0})^{i(n-i)/d}$.
But beware that the elements 
$Y_{i+1\,i+2\,\dots\,n-1,\,i\ i-1\,\dots\,1\,0}$ no longer commute.
\end{remark}

\begin{remark}
From now on, we identify $R$ with the subalgebra of
$k[X_1,\dots,X_{n-1}]$ described in Theorem~\ref{th:R}.
\end{remark}

We can interpret Theorem~\ref{th:R} in terms of finite group schemes. Let $\mu_n$ be
the finite group scheme of $n$th roots of unity, as a subgroup of the
multiplicative group scheme $\bG_m$. The coordinate ring of $\mu_n$ is
$k[\mu_n]=k[x]/(x^n-1)$, with comultiplication $x \mapsto x\otimes x$
and antipode $x \mapsto x^{-1}$.

\begin{theorem}\label{th:mu_n}
The finite group scheme $\mu_n$ acts on $k[X_1,\dots,X_{n-1}]$ via the
map 
\[ \rho^*\colon \mu_n\times \bA^{n-1}(k) \to \bA^{n-1}(k) \] 
given by
\begin{align*}
\rho\colon k[X_1,\dots,X_{n-1}]&\to k[\mu_n]\otimes k[X_1,\dots,X_{n-1}]\\
\rho(X_i)&= x^i\otimes X_i.
\end{align*}
We have $R=k[X_1,\dots,X_{n-1}]^{\mu_n}$ and $\Spec(R)=\bA^{n-1}(k)/\mu_n$.
\end{theorem}
\begin{proof}
The invariants are by definition the polynomials $f$ with
$\rho(f)=1\otimes f$. A monomial satisfies this if and only if $f$ is
of the form described in Theorem~\ref{th:R}, and these span the invariants.
\end{proof}

\begin{remark}\label{rk:zeta}
If $k$ contains $n$ distinct $n$th roots of unity $1,\zeta,
\zeta^2,\dots,\zeta^{n-1}$, then $\mu_n$ is
isomorphic to the constant cyclic group scheme $\bZ/n$ after choosing
a primitive root. The action described in 
Theorem~\ref{th:mu_n} is given by the diagonal matrix
\[ \begin{pmatrix} \zeta & & &  0 \\
&\zeta^2 & & \\
&&\ddots & \\
 0& & & \zeta^{n-1}
\end{pmatrix}. \]
and so $\Spec(R)\cong \bA^{n-1}(k)/(\bZ/n)$.
\end{remark}

\begin{corollary}\label{co:oeis}
Let $\mfp$ be a minimal prime of $C$. Then 
the number of minimal generators for $C/\mfp$ is given by sequence 
\href{http://oeis.org/A096337}{\rm A096337} in
the {\rm OEIS}, which begins as follows:
\begin{center} 
\hfill\rm 0,\ 1,\ 3,\ 6,\ 14,\ 19,\ 47,\ 64,\ 118,\ 165,\ 347,\ 366,\ 826,\
  973,\ 1493,\ 2134,\ \dots \hfill \qed
\end{center}
\end{corollary}
\begin{proof}
Theorem~\ref{th:R} shows that the number of minimal generators is the number
of indecomposable positive solutions to the integer equation 
$x_1+2x_2+\dots+(n-1)x_{n-1}\equiv 0 \pmod n$, which is the
description for this entry in OEIS.
\end{proof}

\begin{remark}
A (weak) lower bound for this number is
given in Dixmier, Erd\H{o}s and
Nicolas~\cite{Dixmier/Erdos/Nicolas:1987a}, where it is denoted $F(n)$.  See also Harris and 
Wehlau~\cite{Harris/Wehlau:2006a}. 
\end{remark}

\begin{proposition}
The ring $R$ is Cohen--Macaulay. It is Gorenstein if and
only if $n$ is odd or $n=2$.
\end{proposition}
\begin{proof}
In this proof, we make use of Theorems~\ref{th:R} and~\ref{th:Zvhat} to
identify $R$ with the integral domain $C/\mathfrak{p}$ and the
subring stated in Theorem~\ref{th:R}. By Theorem \ref{th:mu_n}, since $\mu_n$ is diagonalisable and $R$ is the invariant ring of a regular ring by $\mu_n$, $R$ is Cohen--Macaulay by the Hochster--Roberts theorem \cite{Hochster/Roberts:1974a}.
The Poincar\'e series is described by the non-negative solutions of
\[ x_1+2x_2+\dots+(n-1)x_{n-1}\equiv 0 \pmod{n} \] 
with the $x_i$ having
weight two. It follows that the Poincar\'e series is 
independent of the field of coefficients. So first assume
that $k$ has $n$ distinct $n$th roots of unity. Then we are looking at
the invariant ring of $\bZ/n$ on $k[X_1,\dots,X_{n-1}]$ with
eigenvalues $\zeta,\zeta^2,\dots,\zeta^{n-1}$, as in
Remark~\ref{rk:zeta}. The determinant of this action is 
$\prod_{i=1}^{n-1}\zeta^i=(-1)^{n-1}$, so
the action is contained in $SL(n-1,k)$ if and only if $n$ is odd. By
the theorem of Watanabe~\cite{Watanabe:1974a,Watanabe:1974b}, as long
as $\bZ/n$ contains no pseudoreflections on the $n-1$ dimensional
representation,  the
invariants are Gorenstein if and only if the action is contained in
$SL(n-1,k)$. For $n\ge 3$ there are no pseudoreflections, and for
$n=2$ the invariants $k[X_1^2]$ are Gorenstein. It follows that $R$ is
Gorenstein if and only if $n$ is odd or $n=2$. Over an arbitrary field, we now argue using Theorem~4.4
of Stanley~\cite{Stanley:1978a} that whether or not a graded integral domain is
Gorenstein can be read off from its Poincar\'e series, and we just
observed that this does not depend on the field.
\end{proof}

\begin{example}
If $n=3$ then the Poincar\'e series for $R$ is
\[ \frac{1+t^4+t^8}{(1-t^6)^2}. \]
The numerator is symmetric, so the algebra is Gorenstein. 

On the other
hand, if $n=4$ then the Poincar\'e series for $R$ is
\[ \frac{1-t^2+t^4+t^6}{(1-t^8)(1-t^4)(1-t^2)}. \]
The numerator is not symmetric, so the algebra is not Gorenstein.
\end{example}

We claim that $R$ is a toric ring, meaning that it is a prime ring
whose variety $\Spec(R)$ is an affine toric variety. This will allow
us to deduce some properties of $R$ using the theory of toric
rings. Our notation will conform with Cox, Little
and Schenck~\cite{Cox/Little/Schenck:2011a}, to which we refer for
background. By definition, an 
affine toric variety $V$ is an irreducible affine variety that
contains a torus $T=(k^\times)^m$ as a dense open subset, such that
the action of $T$ on itself extends to an action on $V$.

Let $R'$ be the subring of the Laurent polynomials
$k[X_1,X_1^{-1},\dots,X_{n-1},X_{n-1}^{-1}]$ spanned by the monomials
$X_1^{m_1}\dots X_{n-1}^{m_{n-1}}$ such that $\sum_{i=1}^{n-1}im_i$ is
divisible by $n$. This is a Laurent polynomial ring generated by the
elements 
\begin{equation}
X_1^2X_2^{-1},\quad X_1^3X_3^{-1},\quad \dots,\quad
X_1^{n-1}X_{n-1}^{-1},\quad X_1^n 
\end{equation}
and their inverses. The ring
$k[X_1,X_1^{-1},\dots,X_{n-1},X_{n-1}^{-1}]$ is a free module of rank
$n$ over this subring, with basis $1$, $X_1$, \dots, $X_{n-1}$.

The inclusion $R\to R'$ gives an injective map 
\[ T=\Spec(R') \to \Spec(R) \] 
as a dense
open subset. The action of $T$ on itself extends to an action
on $\Spec(R)$.
There is a one to one correspondence between
characters of $T$ and monomials in $R'$.
Let $\mcA=\mcA_n=\{\mu_1,\dots,\mu_s\}$ be the set of monomials minimally
generating $R$, where $s$ is described in Corollary~\ref{co:oeis}. 
These correspond to characters
$\bX^{\mu_1},\dots,\bX^{\mu_s}$ of $T$.
Then consider the map 
\[ \Phi_\mcA\colon T\to k^s \]
defined by 
\[ \Phi_\mcA(t)=(\bX^{\mu_1}(t),\dots,\bX^{\mu_s}(t)). \]
The Zariski closure $Y_\mcA$ of the image of this map is isomorphic to
$\Spec(R)$. The map $\Phi_\mcA$ induces a map of character lattices 
\[ \hat\Phi_\mcA\colon \bZ^s\to \bZ\mcA\cong\bZ^{n-1}. \]
By~\cite[Proposition~1.1.9]{Cox/Little/Schenck:2011a}, the \emph{toric ideal} of
relations between the $s$ generators $R$ is generated by binomials,
which are a difference of a pair of monomials corresponding to
positive elements of $\bZ^s$ whose differences go to zero in
$\bZ\mcA$. This is a prime ideal.

\begin{theorem}
The integral domain $R$ is normal (i.e., integrally closed in its
field of fractions).
For $n\ge 3$, the divisor class group $\Cl(R)$ (Weil divisors modulo principal
divisors) is isomorphic to $\bZ/n$. The Picard group (Cartier divisors
modulo principal divisors) is $\Pic(R)=0$, every projective $R$-module is free, and $K_0(R)\cong\bZ$.
\end{theorem}
\begin{proof}
If a monomial is in $R'$ and some power of it is in $R$, then the
monomial is in $R$. This is the saturation condition needed to
apply~\cite[Theorem~1.3.5]{Cox/Little/Schenck:2011a}, and conclude
that $\Spec(R)$ is normal.
Applying~\cite[Theorem~4.1.3]{Cox/Little/Schenck:2011a} to $R$, 
we have $\Cl(R)\cong\bZ/n$ for $n\ge 3$ (for $n=2$ we have
$\Cl(R)=0$). Applying~\cite[Proposition~4.2.2]{Cox/Little/Schenck:2011a},
we have $\Pic(R)=0$. Every projective
$R$-module is free by Bruns and
Gubeladze~\cite[Theorem~8.4]{Bruns/Gubeladze:2009a}, and so $K_0(R)\cong\bZ$.
In fact, this even extends from field coefficients to PID
coefficients.
\end{proof}

\section{\texorpdfstring{$\mcN(\tilde A_{n-1})$-modules}{𝓝(Ãₙ₋₁)-modules}}%
\label{se:Amodules}

Recall that $A$ is an affine nilCoxeter algebra of type $\tilde
A_{n-1}$, $C$ is the large commutative subalgebra of
Section~\ref{se:C}, admitting an action of
$\sym{n}$,  and the centre of $A$ is $R=C^{\sym{n}}$.

Even if we didn't already know from Theorem~\ref{th:fd}
that the simple $A$-modules are finite dimensional, we have the
following alternative proof. 

\begin{lemma}[Zariski~\cite{Zariski:1947a}]\label{le:Zariski}
If $F$ is a finitely generated commutative $k$-algebra and $F$ is a
field then $F$ is a finite algebraic extension of $k$.
\end{lemma}
\begin{proof}
A proof can be found in Atiyah and
Macdonald~\cite{Atiyah/Macdonald:1969a}.
\end{proof}

\begin{lemma}
If $S$ is a simple module over a finitely generated commutative
$k$-algebra then $S$ is finite dimensional over $k$. If $k$ is
algebraically closed then $S$ is one dimensional over $k$.
\end{lemma}
\begin{proof}
Zariski's Lemma~\ref{le:Zariski} implies that if $\mcA$ is a finitely generated
commutative $k$-algebra and $\mfm$ is a maximal ideal then $\mcA/\mfm$ is
a finite algebraic field extension of $k$. In particular, if $k$ is
algebraically closed then $\mcA/\mfm$ is one dimensional over $k$.
If $S$ is a simple $\mcA$-module then choosing a non-zero element of $S$
gives a module homomorphism $\mcA\to S$ whose kernel is a maximal ideal $\mfm$,
so $S$ is isomorphic to $\mcA/\mfm$.
\end{proof}

\begin{lemma}\label{le:C/m}
Let $M$ be a non-zero finitely generated $C$-module. Then there is a surjective 
homomorphism from $M$ to $C/\mfm$ for some maximal ideal $\mfm$ of $C$. 
\end{lemma}
\begin{proof}
Since $C$ is Noetherian by Lemma~\ref{le:Cfg/k}, $M$ is Noetherian and has a
maximal proper submodule $M'$. Then $M/M'$ 
is simple, so choosing a generator, it is isomorphic to $C/\mfm$ for
some maximal ideal $\mfm$. Thus we have a surjective homomorphism $M
\to M/M' \cong C/\mfm$.
\end{proof}

\begin{theorem}
Let $S$ be a simple $A$-module. Then $S$ has finite dimension over
$k$. 
\end{theorem}
\begin{proof}
Since $S$ is simple, it is isomorphic to a quotient of $A$. So by
Theorem~\ref{th:fg}, $S{\da}_C$ is finitely generated.
By Lemma~\ref{le:C/m}, there is a non-zero homomorphism from
$S{\da}_C$ to some $C/\mfm$. Then we have
\[ 0 \ne \Hom_C(S{\da}_C,C/\mfm)= \Hom_C({_C}A_A\otimes_A S,C/\mfm)
\cong \Hom_A(S,\Hom_C({_C}A_A,C/\mfm)). \]
Thus $S$ is isomorphic to a submodule of some
$\Hom_C({_C}A_A,C/\mfm)$. By Lemma~\ref{le:Zariski}, $C/\mfm$ is
finite dimensional over $k$. 
By Theorem~\ref{th:fg}, $\Hom_C({_C}A_A,C/\mfm)$ is finite dimensional
over $C/\mfm$ and hence over $k$. Therefore 
 $S$ is finite dimensional over $k$.
\end{proof}

\begin{definition} {\bf (Central character)}\label{def:central-character}
Let $S$ be a simple $A$-module. 
By Schur's lemma, $\End_A(S)$ is a 
division ring, finite dimensional over $k$. By Theorem~\ref{th:centre},
the action of $A$ on $S$ induces a map $\Phi\colon R\to
Z\End_A(S)$ given by $\Phi(r)(s)=rs$.  The image is an integral domain which is a finite
extension of $k$, so it is a field. So the kernel is a maximal ideal $\mfm$
of $R$. We call the map $R\to R/\mfm$
the \emph{central character}  $\chi_S$. The restriction of $S$ to
$R$ is a direct sum of copies of $\chi_S$. Note that the
field $R/\mfm$ is a finitely generated algebra over $k$. By
Lemma~\ref{le:Zariski} $R/\mfm$ is therefore a finite dimensional field
extension of $k$, and if $k$
is algebraically closed then $R/\mfm\cong k$.
\end{definition}

\begin{definition}{\bf (Weight spaces)}
If $M$ is an $A$-module, and $\mfm$ is a maximal
ideal in $C$, let $\tilde k=C/\mfm$, a finite extension of $k$, 
and let $\chi\colon C \to \tilde k$ be the corresponding map and $\mfm_\chi=\mfm$.
Then we write $M_\chi$ for the \emph{weight space} of $M$
corresponding to $\chi$, namely
the sum of the copies of $C/\mfm$ in $M{\da_C}$. This is equal to
$\{x\in M\mid \mfm_\chi x = 0\}$. 
Those $\chi$ with $M_\chi\ne 0$ are called the \emph{weights} of
$M$.   
\end{definition}

\begin{remark}
Note that a non-zero finite dimensional $A$-module has at least one
weight, because its restriction to $C$ has a minimal non-zero
submodule. For a simple module, all its weights are maximal ideals
lying above the central character. An infinite dimensional module need
not have any weights.
\end{remark}

\begin{lemma}\label{le:direct-sum}
The sum of the weight spaces of an $A$-module is a direct sum, and
defines a submodule.
\end{lemma}
\begin{proof}
Let $V$ be an $A$-module, with weight spaces $V_\chi$.
We must show that $V_\chi\cap \sum_{\chi'\ne\chi}V_{\chi'}=0$. Suppose
that 
\begin{equation}\label{eq:vchi}
v_\chi=\sum_{\chi'}v_{\chi'}
\end{equation} 
with $v_\chi\in
V_\chi$ and a finite number of non-zero $v_{\chi'} \in V_{\chi'}$. For
each of these $\chi'$, choose an element that
is in $\mfm_{\chi'}$ but not in $\mfm_\chi$. Multiplying these
together, we obtain an element $x\in C$ with $x\not\in\mfm_\chi$ but
for each $\chi'$ we have $x\in\mfm_{\chi'}$. Then each $xv_{\chi'}=0$,
and so multiplying \eqref{eq:vchi} by $x$, we have $xv_\chi=0$ and
hence $v_\chi=0$.
\end{proof}

\begin{lemma}\label{le:chiZvne0}
Given a maximal ideal $\mfm$ of $C$, 
we can choose $v\in\bZ^{n-1}$ such that $Z_v\not\in\mfm$.
If $v=(r_1,\dots,r_n)$ satisfies $r_i\ne r_{i+1}$ then $\bar
w_i(\mfm)\ne\mfm$. Conversely, if $\bar w_i(\mfm)\ne\mfm$ then there exists
such a $v=(r_1,\dots,r_n)$ with $Z_v\not\in\mfm$ and $r_i\ne r_{i+1}$.
\end{lemma}
\begin{proof}
Since $\mfm$ is a proper subspace of $C$, and the $Z_v$ span $C$,
there exists $v=(r_1,\dots,r_n)\in\bZ^{n-1}$ with $Z_v\not\in\mfm$. If $r_i\ne
r_{i+1}$ then by Lemma~\ref{le:ZvZv'} we have $Z_vZ_{\bar
  w_i(v)}=0$. Since $Z_v\not\in\mfm$, it follows that $Z_{\bar
  w_i(v)}\in\mfm$. But $Z_{\bar w_i(v)}\not\in\bar w_i(\mfm)$, and so
$\bar w_i(\mfm)\ne\mfm$. Conversely, suppose that for all
$v=(r_1,\dots,r_n)\in\bZ^{n-1}$ with $Z_v \not\in\mfm$, we have
$r_i=r_{i+1}$. Then the elements $Z_v$ with $r_i\ne r_{i+1}$ are in
$\mfm$. So $\mfm$ has a $k$-basis consisting of the elements $Z_v$
with $r_i\ne r_{i+1}$, together with some linear combinations of the
$Z_v$ with $r_i=r_{i+1}$. The former basis elements are permuted by $\bar w_i$, while
the latter are fixed by $\bar w_i$, so $\bar w_i(\mfm)=\mfm$.
\end{proof}

\begin{lemma}\label{le:permute-weights}
Let $M$ be a finite dimensional $A$-module, $\mfm$ be a maximal
ideal in $C$ corresponding to $\chi\colon C \to \tilde k$. Then for
$0\le i\le n-1$ we have
$Y_iM_\chi\subseteq M_{\bar w_i(\chi)}$ (cf.\ Definition~\ref{def:wbar}).
\end{lemma}
\begin{proof}
For $x\in M_\chi$ and $v=(r_1,\dots,r_n)\in\bZ^{n-1}$, using
Lemma~\ref{le:YiZv} we have  
\begin{equation*} 
Z_vY_ix=Y_iZ_{\bar w_i(v)}x=\chi(Z_{\bar w_i(v)})Y_ix = \bar w_i(\chi)(Z_v)Y_ix,
\end{equation*}
and so $Y_ix$ is in the $\bar w_i(\chi)$ weight space $M_{\bar w_i(\chi)}$.
\end{proof}

\begin{lemma}\label{le:vwiwiv}
Given $i$ with $0\le i\le n-1$ and $v=(r_1,\dots,r_n)\in\bZ^{n-1}$
with $\bar w_i(v)\ne v$, i.e.,
\[ \begin{cases} r_i\ne r_{i+1} & 1\le i\le n-1 \\ r_n\ne r_1 &
    i=0 \end{cases} \]
we have $Y_{vw_i}Y_{w_iv}=Z_v^2$.
\end{lemma}
\begin{proof}
This follows immediately from Proposition~\ref{pr:w'w''}.
\end{proof}

\begin{theorem}\label{th:weights}
If $\chi$ is a weight of an $A$-module then so is $\sigma(\chi)$ for
every $\sigma\in\sym{n}$, and the weight spaces for $\chi$ and
$\sigma(\chi)$ have the same dimension over $C/\mfm= \tilde k$ and
$C/\bar
w_i(\mfm)\cong\tilde k$ respectively. In particular, they have the same dimension
over $k$.
\end{theorem}
\begin{proof}
It suffices to show that each $\bar w_i(\chi)$ is
a weight, since the $\bar w_i$ generate $\sym{n}$. 
If $\bar w_i(\chi)=\chi$, this is obvious, so suppose that 
$\bar w_i(\chi)\ne \chi$. Using Lemma~\ref{le:chiZvne0}, choose
$v=(r_1,\dots,r_n)$ such that $r_i\ne r_{i+1}$ and
$Z_v\not\in\mfm_\chi$, so that $\chi(Z_v)\ne 0$. By
Lemma~\ref{le:permute-weights}, the operations 
$Y_{vw_i}$ and $Y_{w_iv}$ swap the spaces $M_{\chi}$ and
$M_{\bar w_i(\chi)}$, because 
\[ Y_{w_iv}M_\chi\subseteq M_{\overline{w_iv}(\chi)}
=M_{\bar w_i(\chi)},\qquad
Y_{vw_i}M_{\bar w_i(\chi)}\subseteq M_{\overline{vw_i}\bar
  w_i(\chi)}=M_\chi, \]
and by Lemma~\ref{le:vwiwiv} the composite
equals the action of $Z_v^2$. 
This is equal to multiplication by
$\chi(Z_{v})^2$, which is a non-zero scalar in $\tilde k$ by
choice of $v$. Similarly, the composite $Y_{w_iv}Y_{vw_i}$ is equal
to multiplication by $\bar w_i(\chi)(Z_{\bar w_i(v)})^2$.  It follows that they induce an 
isomorphism $M_{\chi}\cong M_{\bar w_i(\chi)}$. Thus all the
weight spaces $M_{\sigma(\chi)}$ have the same dimension over $C/\mfm\cong\tilde
k$ and over $C/\bar w_i(\mfm)=\bar w_i(C/\mfm)\cong \tilde k$
respectively, and hence also the same dimension over $k$. 
\end{proof}

\begin{corollary}\label{co:weights}
Let $S$ be a simple $A$-module. Then $S$ is the direct sum of its
weight spaces, and these form a single orbit of $\sym{n}$,
lying above the central character of the module.
\end{corollary}
\begin{proof}
By Theorem~\ref{th:weights}, the weights of a simple module $S$ are a
union of orbits of $\sym{n}$.
If $\chi$ is a weight of $S$, then by Lemma~\ref{le:permute-weights}
the sum of the weight spaces for weights in the orbit of $\chi$ forms
a submodule. By Lemma~\ref{le:direct-sum}, the sum is direct.
\end{proof}

Next, we discuss the two extreme cases: the trivial orbit and regular orbits.

\begin{definition}
Let $V$ be an $A$-module. We say that $V$ is a \emph{trivial module}
if $V\cong A/J$ is one dimensional and all the $Y_i$ act as zero. This 
has just one weight space, corresponding to the \emph{zero weight}, 
where the maximal ideal $\mfm_0$ is the one generated by monomials in 
$C$ of positive degree. 
\end{definition}

\begin{theorem}{\bf (Trivial module)}\label{th:trivial}
\begin{enumerate}
\item 
$C/\mfm_0C$ is a finite dimensional local algebra with residue field $k$. 
\item
$A/A\mfm_0A$ is a finite dimensional local algebra with residue field $k$.
\item
If $S$ be a simple $A$-module with a zero 
weight, then $S$ is a trivial module.
\end{enumerate}
\end{theorem}
\begin{proof}
(1) This follows from Theorem~\ref{th:Cfg/Csigma}.

(2) By Theorem~\ref{th:fg}, $A/A\mfm_0A$ is a finitely generated
$C$-module with $\mfm_0$ acting trivially. Therefore $A/A\mfm_0A$ is a
finitely generated $C/\mfm_0C$-module, and hence $A/A\mfm_0A$ is
finite dimensional over $k$ by (1). 
By Lemma~\ref{le:Ywd}, if $w$ is a word of positive length, some power
of $Y_w$ is in $\mfm_0C\subseteq A\mfm_0A$. Let $Q$ be the ideal of $A/A\mfm_0A$ spanned by $Y_w+A\mfm_0A$ with $w\neq 1$. Since $A\mfm_0A$ is a homogeneous ideal of $A$, $Q$ inherits the grading and hence $Q=\bigoplus_{i\ge 1} Q_i$. Since $Q$ is finite-dimensional over $k$, there exists $N$ such that  $Q_i=0$ for all $i>N$. So $Q^{N+1}=0$. Thus any maximal ideal of
$A/A\mfm_0A$ contains all $Y_w+A\mfm_0A$ and hence it is
$J/A\mfm_0A$. Thus $A/A\mfm_0A$ is local.

(3) By Lemma~\ref{le:permute-weights}, the zero weight space 
of $S$ is a nonzero submodule of $S$, and is hence the whole of $S$. Thus
$S$ is a simple module for $A/A\mfm_0A$, and then by (2), it is 
isomorphic to the trivial module for $A$.
\end{proof}

\begin{theorem}\label{th:n!}
Suppose that $\mfm$ is a maximal ideal  of $C$,
with corresponding map 
$\chi\colon C\to \tilde k$, and suppose that the stabiliser of $\chi$
in $\sym{n}$ is trivial. If $S$ is a simple $A$-module and $\chi$ is a weight of $S$
then each weight space $S_{\bar w(\chi)}$ is one dimensional, and $S$ has
dimension $n!$ over $\tilde k$.
\end{theorem}
\begin{proof}
By Corollary~\ref{co:weights}, $S$ is the direct sum of its weight
spaces, and the weight spaces form a regular orbit of $\sym{n}$.
Let $\pi\colon S \to S_{\chi}$ be the projection onto the $\chi$
weight space. Then
given any reduced word $w$ in the $w_i$, by
Lemma~\ref{le:permute-weights} we have $Y_wS_\chi\subseteq S_{\bar
  w(\chi)}$, and so unless $\bar w=1$ we have $\pi(Y_wS_\chi)=0$.
The words $w$ with $\bar w=1$ are the elements of $\bZ^{n-1}$, so the
conclusion is that $\pi(Y_wS_\chi)=0$ unless $Y_w\in C$. Thus
for $x\in S_\chi$ we have $\pi(A.x)=C.x=\tilde k.x$.
In other words, for $0\ne x\in S_\chi$, the $\chi$ weight space of the submodule of $S$
generated by $x$ is one dimensional over $\tilde k$. Since $S$ is
simple, we conclude that the weight spaces of $S$ are one dimensional
over $\tilde k$, using Theorem~\ref{th:weights}. Their sum is direct, so $S$ has dimension $n!$ over
$\tilde k$.
\end{proof}

\begin{theorem}\label{th:n!2}
The algebra $A$ is a prime affine PI algebra of PI degree $n!$.
The simple $A$-modules have dimension at most $n!$ over a finite
extension field $\tilde k$ of $k$. For 
a dense open subset of characters of $R$, there is exactly one 
simple module with that central character, and its dimension over
$\tilde k$ is $n!$.
\end{theorem}
\begin{proof}
We refer the reader to Chapter~13 of McConnell and
Robson~\cite{McConnell/Robson:2001a}, Chapter~6 of Rowen~\cite{Rowen:1988b}, and
Chapter~III.1 of Brown and Goodearl~\cite{Brown/Goodearl:2002a},
especially Theorem~III.1.6 there,
for the general theory underlying this theorem.
 
By Theorem~\ref{th:prime} and Corollary~\ref{co:PI}, $A$ is a prime
affine PI algebra. Let $m$ be its PI degree, let $g_m$ be the
central polynomial defined in~\cite[\S13.5.11]{McConnell/Robson:2001a}.
Recall from~\cite[\S13.5.2]{McConnell/Robson:2001a} that a central
polynomial is a polynomial with integer coefficients whose values
lie in the centre but are not constant. Thus there exists
$0\ne c\in g_m(A)$, and such a $c$ is in the centre of $A$ but not in
$k$. Recall from
Theorem~\ref{th:centre} that the centre of $A$ is equal to $R$, which
by Corollary~\ref{co:domain} is an integral domain. It follows that
$\Spec(R)$ is an irreducible affine variety, and
the hypersurface complement
$\Spec(R[c^{-1}])$ is a dense open subset of $\Spec(R)$. 

By~\cite[Proposition~13.7.4]{McConnell/Robson:2001a}, 
all primes of $A[c^{-1}]$ are regular. The ring $A[c^{-1}]$ is a prime
affine PI algebra, so by the Artin--Procesi theorem
(Artin~\cite{Artin:1969a}, Procesi~\cite{Procesi:1972a}, see also
Amitsur~\cite{Amitsur:1973a} and
\cite[Theorem~13.7.14]{McConnell/Robson:2001a}),  $A[c^{-1}]$ is an
Azumaya algebra of PI degree $m$.
It follows that for $\mfp$ in the dense subset $\Spec(R[c^{-1}])$ of
$\Spec(R)$, the PI degree of $A_\mfp$ is 
equal to $m$. But by 
Theorem~\ref{th:n!}, there is a dense open 
subset of the maximal ideals $\mfm$ of $R$ such that the PI
degree of $A/\mfm A$ is equal to $n!$. These dense open subsets must
intersect, so it follows that $m=n!$.
Finally, if $\mfm$ is any maximal ideal of $R$, $A/\mfm A$ is an
affine PI algebra of degree at most $n!$, and so its simple modules
have dimension at most $n!$.
\end{proof}

\begin{remark}\label{rk:Azumaya-locus}
The set of prime ideals $\mfp$ of $R$ such that $A/\mfp A\cong
\Mat_{n!}(R/\mfp)$ is the Azumaya locus of $A$, see for example
Brown and Goodearl~\cite{Brown/Goodearl:1997a,Brown/Goodearl:2002a}. 
This is a dense open subset, 
equal to the set of primes where the map $\Spec(C)\to\Spec(R)$ is
unramified. These are the images of points in $\Spec(C)$ with
trivial stabiliser in $\sym{n}$.
\end{remark}

\begin{corollary}
Let $K$ be the field of fractions of the integral domain $R=Z(A)$, and
let $A_K=K\otimes_RA$. Then
$A_K$ is isomorphic to $\Mat_{n!}(K)$, and $A$ is an order in $A_K$.
\end{corollary}
\begin{proof}
This follows from Theorem~\ref{th:n!2} using Posner's theorem~\cite{Posner:1960a}. 
See also~\cite[Theorem~13.6.5]{McConnell/Robson:2001a}.
\end{proof}

\begin{corollary}
The nilCoxeter algebra $A$ is a bounded Goldie ring.
\end{corollary}
\begin{proof}
This follows from Theorem~\ref{th:n!2} 
using~\cite[Corollary~13.6.6]{McConnell/Robson:2001a}.
\end{proof}

\section{\texorpdfstring{Example: The case $\tilde A_1$}
{Example: The case Ã₁}}\label{se:A1tilde}

The nilCoxeter algebra $A=\mcN(\tilde A_1)$ of type $\tilde A_1$ has two generators, $Y_0$
and $Y_1$, subject only to the relations $Y_0^2=Y_1^2=0$. It is the
algebra whose representation theory is described in
Ringel~\cite{Ringel:1975a}, so its finite dimensional modules are
strings and bands.

The subalgebra $C$ is generated by $Z_v=Y_{10}$ and
$Z_{-v}=Y_{01}$ where $v=(-1,1)$, swapped by the action of $\sym{2}$, and
annihilating each other. Recall from Notation~\ref{not:Y} that
$Y_{10}$ denotes $Y_1Y_0$, and so on. The
$\sym{2}$-invariants are $R=C^{\sym{2}}=k[z]$ where $z=Y_{10}+Y_{01}=Z_v+Z_{-v}$.
There are two minimal primes, $\mfp=(Y_{10})$ and $\mfp'=(Y_{01})$ in
$C$, which consist respectively of the circles to the right and to the left of
the identity of $A$  (the star) in the following diagram.
\[ \xymatrix@=0.5cm{&\cdot\ar@{..}[l]&
\smash{\circ}\ar@[red][l]&\cdot\ar@[blue][l]&\smash{\circ}\ar@[red][l]&
\cdot\ar@[blue][l]&\smash{\bigstar}\ar@[red][l]\ar@[blue][r]&
\cdot\ar@[red][r]&\smash{\circ}\ar@[blue][r]&\cdot\ar@[red][r]
&\smash{\circ}\ar@[blue][r]&\cdot\ar@{..}[r]&} \]
In this diagram, the vertices represent monomial basis elements. The
blue arrows are for left multiplication by $Y_0$, red for $Y_1$.\medskip

Let $M$ be a non-trivial simple $A$-module. Since the non-trivial
characters of $C$ are permuted regularly by $\sym{2}$, by
Theorems~\ref{th:weights} and~\ref{th:trivial} the weights of $M$ are
$\chi$ and $-\chi$ for some non-trivial character $\chi$ of $C$. By
Theorem~\ref{th:n!} the weight spaces are one dimensional over 
$\tilde k=C/\mfm_\chi\cong C/\mfm_{-\chi}\cong
R/R\cap\mfm_\chi$, and $\tilde k\cong k[z]/(f(z))$ for some irreducible monic polynomial
$f(z)\ne z$. The module $M$ may therefore be depicted by the following
diagram:
\begin{equation}\label{eq:2dim-simple-tilde}
\xymatrix{k[z]/(f(z)) \ar@/^{4ex}/[r]^{Y_0}_1 &
    k[z]/(f(z))\ar@/^{4ex}/[l]^{Y_1}_z} 
\end{equation}
If $k$ is algebraically closed then the irreducible monic polynomials
other than $z$ are $z-\lambda$ with $\lambda\ne 0$. Thus $M$ is two
dimensional over $k$, with diagram
\begin{equation}\label{eq:2dim-simple}
\xymatrix{k \ar@/^{3ex}/[r]^{Y_0}_1 &
    k\ar@/^{3ex}/[l]^{Y_1}_\lambda} 
\end{equation}
Thus every central character corresponds to a two dimensional simple
module, except the trivial character, corresponding to the trivial one
dimensional module.

Another way to see this with less machinery is as follows.
Let $M$ be a non-trivial simple $A$-module.
Let $V_0=\Ker(Y_1)$ and $V_1=\Ker(Y_0)$. Notice that these are
non-zero, because $V_0$ contains the image of $Y_1$, so if $V_0$ is
zero then $Y_1$ acts trivially, and then it follows that $Y_0$ also
acts trivially because $M$ is simple. Then $Y_0V_0\subseteq V_1$ and
$Y_1V_1\subseteq V_0$. Thus $V_0$ and $V_1$ span a submodule of $M$,
and if $M$ is simple then they span $M$. Furthermore, this sum is
direct since $V_0\cap V_1$ is a submodule of $M$.
The maps $Y_0\colon V_0\to V_1$ and
$Y_1\colon V_1\to V_0$ are injective, because $Y_0$ and $Y_1$ both
act as zero on the kernels, producing a trivial submodule if the
intersection of the
kernels were non-zero. If $Y_0\colon V_0\to V_1$ is not
surjective, then $V_0\oplus Y_0V_0$ is a proper submodule, and
similarly for $Y_1$. Thus $Y_0\colon V_0\to V_1$ and $Y_1\colon V_1\to
V_0$ are isomorphisms. The
composite $Y_{10}\colon V_0\to V_0$ is an automorphism. In fact it is
the automorphism induced by the central element $z=Y_{10}+Y_{01}$ since $Y_{01}$ acts as zero on $V_0$.

The  action of $z$ on $V_0$ makes it a simple
$k[z]$-module with non-zero action of $z$. The simple $k[z]$-modules are finite
dimensional, and are
parametrised by irreducible polynomials $f(z)$. 
So the simple $A$-modules are as given
in~\eqref{eq:2dim-simple-tilde}, as $f(z)$ runs over the irreducible
polynomials other than $z$.
If $k$ is
algebraically closed, the irreducible polynomials are linear, and the
simples form a one parameter family of two dimensional simples indexed
by the possible non-zero eigenvalues of $z$, so the simples are as
given in~\eqref{eq:2dim-simple}.\bigskip

We remark that from the above, Conjecture~\ref{conj:multinomial} is easily seen to be true in the
example of $\tilde A_1$. In the next two sections, we shall see that
it is also true in the cases of $\tilde A_2$ and $\tilde A_3$.

\section{\texorpdfstring{Example: The case $\tilde A_2$}
{Example: The case Ã₂}}\label{se:A2tilde}

The nilCoxeter algebra $A$ of type $\tilde A_2$ has three generators,
$Y_0$, $Y_1$ and $Y_2$, subject to the relations
\[ Y_0^2=Y_1^2=Y_2^2=0,\quad Y_{010}=Y_{101},\quad
  Y_{121}=Y_{212},\quad Y_{202}=Y_{020} \]
(Notation~\ref{not:Y}). 
A picture to enable the reader to visualise 
this algebra may be found in Figure~\ref{fig:A2tilde} on
page~\pageref{fig:A2tilde}. In this figure, the vertices represent
monomial basis elements. The
blue arrows are for left multiplication by $Y_0$, red for $Y_1$
and green for $Y_2$. The star is the identity element. 

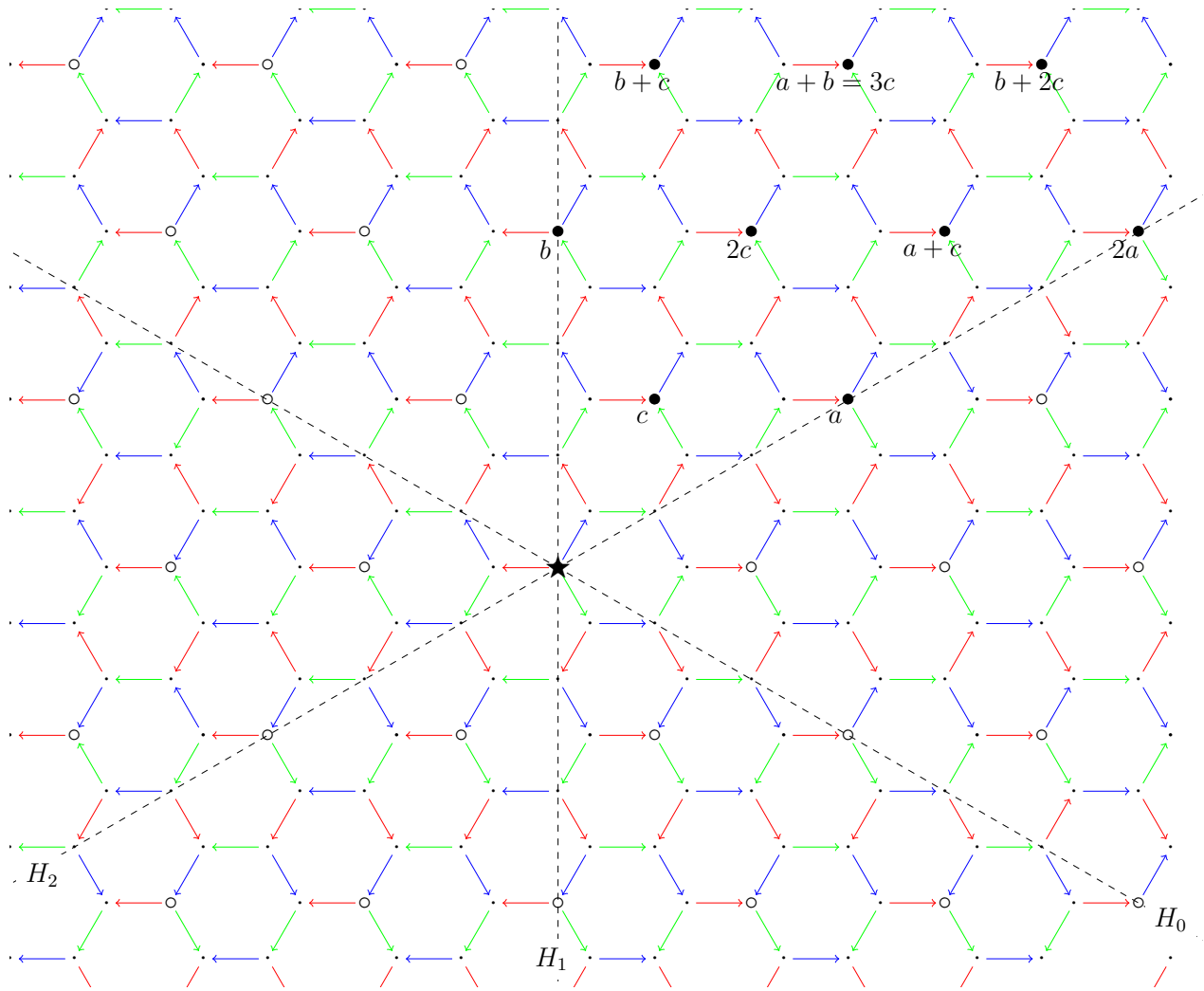
\begin{figure}[h]
\centering
\scalebox{0.9}{
\begin{tikzpicture}[yscale={sqrt(3)/2}]
    \clip (-8.5,-7.5) rectangle (10,10);
    
    \tikzset{
    every path/.style={
      shorten <=4pt,
      shorten >=4pt
    }}
    %horizontal arrows on right half
    \foreach \i in {0,3,6}
    \foreach \j in {-6,0,6}
    {\draw [->,red] (0.5+\i,3+\j)--(1.5+\i,3+\j);
    \draw [->,green] (0.5+\i,1+\j)--(1.5+\i,1+\j);
    \draw [->,blue] (0.5+\i,-1+\j)--(1.5+\i,-1+\j);
    \draw [->,red] (2+\i,\j)--(3+\i,\j);
    \draw [->,blue] (2+\i,2+\j)--(3+\i,2+\j);
    \draw [->,green] (2+\i,4+\j)--(3+\i,4+\j);}
    
    %horizontal arrows on left half
    \foreach \i in {-9,-6,-3}
    \foreach \j in {-6,0,6}
    {\draw [<-,red] (0.5+\i,3+\j)--(1.5+\i,3+\j);
    \draw [<-,green] (0.5+\i,1+\j)--(1.5+\i,1+\j);
    \draw [<-,blue] (0.5+\i,-1+\j)--(1.5+\i,-1+\j);
    \draw [<-,red] (2+\i,\j)--(3+\i,\j);
    \draw [<-,blue] (2+\i,2+\j)--(3+\i,2+\j);
    \draw [<-,green] (2+\i,4+\j)--(3+\i,4+\j);}
    
    %upper half of \ell_2
    \foreach \i in {-2,-1,0,1,2}
    \foreach \j in {0,1,2,3}
    {\draw [->,red] (0.5+4.5*\i-1.5*\j,1+3*\i+3*\j)-- (0+4.5*\i-1.5*\j,2+3*\i+3*\j);
    \draw [->,green] (2+4.5*\i-1.5*\j,2+3*\i+3*\j)-- (1.5+4.5*\i-1.5*\j,3+3*\i+3*\j);
    \draw [->,blue] (-1+4.5*\i-1.5*\j,0+3*\i+3*\j)-- (-1.5+4.5*\i-1.5*\j,1+3*\i+3*\j);
    \draw [->,red] (2+4.5*\i-1.5*\j,4+3*\i+3*\j)-- (1.5+4.5*\i-1.5*\j,5+3*\i+3*\j);
    \draw [->,green] (3.5+4.5*\i-1.5*\j,5+3*\i+3*\j)-- (3+4.5*\i-1.5*\j,6+3*\i+3*\j);
    \draw [->,blue] (0.5+4.5*\i-1.5*\j,3+3*\i+3*\j)-- (0+4.5*\i-1.5*\j,4+3*\i+3*\j);}
    
    %lower half of \ell_2
    \foreach \i in {-3,-2,-1,0,1,2}
    \foreach \j in {-3,-2,-1}
    {\draw [<-,red] (0.5+4.5*\i-1.5*\j,1+3*\i+3*\j)-- (0+4.5*\i-1.5*\j,2+3*\i+3*\j);
    \draw [<-,green] (2+4.5*\i-1.5*\j,2+3*\i+3*\j)-- (1.5+4.5*\i-1.5*\j,3+3*\i+3*\j);
    \draw [<-,blue] (-1+4.5*\i-1.5*\j,0+3*\i+3*\j)-- (-1.5+4.5*\i-1.5*\j,1+3*\i+3*\j);
    \draw [<-,red] (2+4.5*\i-1.5*\j,4+3*\i+3*\j)-- (1.5+4.5*\i-1.5*\j,5+3*\i+3*\j);
    \draw [<-,green] (3.5+4.5*\i-1.5*\j,5+3*\i+3*\j)-- (3+4.5*\i-1.5*\j,6+3*\i+3*\j);
    \draw [<-,blue] (0.5+4.5*\i-1.5*\j,3+3*\i+3*\j)-- (0+4.5*\i-1.5*\j,4+3*\i+3*\j);}
    
    %upper half of \ell_3
    \foreach \i in {-2,-1,0,1,2}
    \foreach \j in {0,1,2,3}
    {\draw [->,red] (-1.5-4.5*\i+1.5*\j,1+3*\i+3*\j)-- (-1-4.5*\i+1.5*\j,2+3*\i+3*\j);
    \draw [->,blue] (-4.5*\i+1.5*\j,3*\i+3*\j)-- (0.5-4.5*\i+1.5*\j,1+3*\i+3*\j);
    \draw [->,green] (1.5-4.5*\i+1.5*\j,-1+3*\i+3*\j)-- (2-4.5*\i+1.5*\j,3*\i+3*\j);
    \draw [->,red] (1.5-4.5*\i+1.5*\j,1+3*\i+3*\j)-- (2-4.5*\i+1.5*\j,2+3*\i+3*\j);
    \draw [->,blue] (3-4.5*\i+1.5*\j,3*\i+3*\j)-- (3.5-4.5*\i+1.5*\j,1+3*\i+3*\j);
    \draw [->,green] (4.5-4.5*\i+1.5*\j,-1+3*\i+3*\j)-- (5-4.5*\i+1.5*\j,3*\i+3*\j);}
    
    %lower half of \ell_3
    \foreach \i in {-2,-1,0,1,2}
    \foreach \j in {-3,-2,-1}
    {\draw [<-,red] (-1.5-4.5*\i+1.5*\j,1+3*\i+3*\j)-- (-1-4.5*\i+1.5*\j,2+3*\i+3*\j);
    \draw [<-,blue] (-4.5*\i+1.5*\j,3*\i+3*\j)-- (0.5-4.5*\i+1.5*\j,1+3*\i+3*\j);
    \draw [<-,green] (1.5-4.5*\i+1.5*\j,-1+3*\i+3*\j)-- (2-4.5*\i+1.5*\j,3*\i+3*\j);
    \draw [<-,red] (1.5-4.5*\i+1.5*\j,1+3*\i+3*\j)-- (2-4.5*\i+1.5*\j,2+3*\i+3*\j);
    \draw [<-,blue] (3-4.5*\i+1.5*\j,3*\i+3*\j)-- (3.5-4.5*\i+1.5*\j,1+3*\i+3*\j);
    \draw [<-,green] (4.5-4.5*\i+1.5*\j,-1+3*\i+3*\j)-- (5-4.5*\i+1.5*\j,3*\i+3*\j);}

    \draw [dashed,-] (0,-10)--(0,10);
    \draw [dashed,-] (-13.5,9)--(13.5,-9);
    \draw [dashed,-] (-13.5,-9)--(13.5,9);
    \node [fill=white] at (-.1,-7) {$H_1$};
    \node [fill=white] at (9.5,-6.3) {$H_0$};
    \node [fill=white] at (-8,-5.5) {$H_2$};
    \node at (0,0) {$\bigstar$};
    
    %insert dots
    \foreach \x in {-9,-6,-3,0,3,6,9}
    \foreach \y in {-12,-6,0,6,12}
    {\node at (\x+2,\y) {$\cdot$};
    \node at (\x+.5,\y-3) {$\cdot$};
    \node at (\x+.5,\y+1) {$\cdot$};
    \node at (\x+1.5,\y+1) {$\cdot$};
    \node at (\x,\y+2) {$\cdot$};
    \node at (\x-1,\y+2) {$\cdot$};
    \node at (\x+.5,\y-3) {$\cdot$};
    \node at (\x+.5,\y+5) {$\cdot$};
    \node at (\x+1.5,\y+5) {$\cdot$};
    \node at (\x,\y+4) {$\cdot$};
    \node at (\x-1,\y+4) {$\cdot$};}
    
    %insert circles
    \foreach \x in {-9,-6,-3,0,3,6,9}
    \foreach \y in {-12,-6,0,6,12}
    {\node at (\x,\y) {$\circ$};
    \node at (\x+1.5,\y+3) {$\circ$};}
    
    %insert bullets
    \foreach \x/\y/\label in {1.5/3/c, 4.5/3/a, 0/6/b, 3/6/2c, 6/6/a+c, 9/6/2a, 1.5/9/b+c, 4.5/9/a+b=3c, 7.5/9/b+2c}
    {\node at (\x-.2,\y-.3) {$\label$};
    \node at (\x,\y) {$\bullet$};}
    
  \end{tikzpicture}}
\caption{The affine nilCoxeter algebra of type $\tilde A_2$}\label{fig:A2tilde}
\end{figure}

The commutative subalgebra $C$ may be described as follows.
The quotient by the prime $\mfp_0$ has a basis
consisting of the images of the $Z_v$ of
Section~\ref{se:C} 
where $v=(r_1,r_2,r_3)$ satisfies $r_1\le
r_2\le r_3$. Let
\[ a=(-2,1,1),\qquad b=(-1,-1,2),\qquad c=(-1,0,1) \]
and notice that $a+b=3c$. 
Then 
\[
Z_a=Y_{120}^2, \qquad
Z_b=Y_{210}^2, \qquad
Z_c=Y_{1210}=Y_{2120}.
\] In Figure \ref{fig:A2tilde}, the reflection about the line $H_i$ indicate the action of the symmetric group $\sym{3}$ on $C$ given by the transposition $\bar w_i$ for $i=0,1,2$. For example, $\bar w_1Z_a=Y_{201}^2$. The images of these generate $C/\mfp_0$, and satisfy $Z_aZ_b=Z_c^3$.
These and their images under $\sym{3}$ generate $C$. There are six minimal
primes in $C$, which are the six images of $\mfp_0$ under
$\sym{3}$.

The centre $R=C^{\sym{3}}$ is spanned by the elements
\begin{align*} 
\alpha=\hat Z_a&=Y_{120}^2+Y_{201}^2+Y_{012}^2
=(Y_{120}+Y_{201}+Y_{012})^2, \\
\beta=\hat Z_b&=Y_{210}^2+Y_{021}^2+Y_{102}^2
=(Y_{210}+Y_{021}+Y_{102})^2, \\
\gamma=\hat Z_c&=Y_{1210}+Y_{2021}+
Y_{0102}+Y_{0121}+Y_{1202}+Y_{2010}\\
&=(Y_{12}+Y_{20}+Y_{01})(Y_{21}+Y_{02}+Y_{10}). 
\end{align*}
The factors in the factorisation of $\gamma$ here commute.
We have 
\[ Z(A)=R=k[\alpha,\beta,\gamma]/(\alpha\beta-\gamma^3). \]
This is a normal integral domain of Krull dimension two. According to 
Theorems~\ref{th:R} and \ref{th:mu_n}, it can be identified with the ring
$k[X_1^3,X_2^3,X_1X_2]=k[X_1,X_2]^{\mu_3}$. 

Now let $M$ be a simple $A$-module. There are four cases for the
nature of the central character $\chi:R\to R/\mfm=\tilde k$ associated to $M$
(Definition~\ref{def:central-character}): 
\begin{enumerate}
\item $\alpha$ and $\beta$ both go to zero, 
\item $\beta$ goes to zero but $\alpha$ does not,
\item $\alpha$ goes to zero but $\beta$ does not,
\item neither $\alpha$ nor $\beta$ goes to zero.
\end{enumerate}
Note that $\gamma$ goes to zero in the first three cases but not in
the last case.

In case (1), where $\alpha$ and $\beta$ both go to zero, this corresponds to the
origin in $\Spec(R)$, which is the singular point. 
This is the zero central character, and the only simple module 
in this case is the trivial module $k$. By a Gr\"obner--Shirshov basis computation using Y{\i}lmaz, \"Ozel and
Ustao\u{g}lu~\cite[Theorem~21]{Yilmaz/Ozel/Ustaoglu:2014a}, and checked
by {\sc Magma} in several characteristics for reassurance,
$A/(\alpha,\beta,\gamma)A$ is a local
algebra of dimension $60$ whose
socle series and radical series are equal, with
layers of dimension 
\[ 1,\quad 3,\quad 6,\quad 9,\quad 11,\quad 12,
\quad 10,\quad 6,\quad 2. \]
In particular, this algebra is not self-injective. We shall see that
this contrasts with 
the algebras $A/\mfm A$ for $\mfm$ not equal to the origin.

In case (4), where $\gamma$ does not go to zero, the weights of $M$ form a regular orbit under the action of $\sym{3}$, so we are in the
situation of Theorem~\ref{th:n!}. So $M$ has six weight spaces, all of
dimension one over $\tilde k$, and $M$ is a $6$-dimensional $\tilde
k$-module. In this case, $A/\mfm A\cong \Mat_6(\tilde k)$, see
Remark~\ref{rk:Azumaya-locus}. So case (4) corresponds to the Azumaya
locus of $A$.

This leaves cases (2) and (3). These are similar, so we shall treat case (2). In this case, the stabilisers of the weights are conjugates of $\sym{2}=\langle\bar w_2\rangle$. Thus there are three
weight spaces, permuted transitively by $\sym{3}$.

Let us examine the general $A$-module $M$ on which $R$ acts via a
central character that is non-zero on $\alpha$ and zero on $\beta$ and
$\gamma$, so that there are three weight spaces lying above the
central character, permuted transitively by $\sym{3}$.

Let $M_1$ be the weight space stabilised by $\bar w_2=(2, 3)$ so that
$M_1$ corresponds to $C/\mfp_0$, acted on by $Z_a$, $Z_b$ and
$Z_c$. Thus $Z_a$ acts on $M_1$ as a non-zero scalar $\lambda$,
and $Z_b$ and $Z_c$ act as zero. Since $Y_1Z_a=0$, we have $Y_1M_1=0$. By Lemma~\ref{le:permute-weights}, we have $Y_2M_1\subseteq M_1$ and
$Y_0M_1\subseteq M_0$, where $M_2$ is the image of $M_1$ under $\bar w_1=(1,2)$, namely the weight space stabilised by $\bar w_0=(1,3)$, and $M_0$ is the image of $M_1$ under $\bar w_0$, namely the weight space stabilised by $\bar w_1$. Then for $i=0,1,2$, read modulo 
three, we have $Y_iM_i=0$, 
$Y_{i+1}M_i\subseteq M_i$, and $Y_{i-1}M_i\subseteq M_{i-1}$. The maps $Y_{i+1}\colon M_i\to M_i$
square to zero. Since $\hat Z_a$ acts as multiplication by $\lambda$, the maps $Y_{i-1}\colon M_i\to M_{i-1}$ are isomorphisms. We can use the
isomorphisms $M_2\xrightarrow{Y_1}M_1\xrightarrow{Y_0}M_0$ to adjust
bases so that they are the identity, and then $Y_2\colon M_0\to M_2$
acts as an automorphism $u$ that squares to multiplication by
$\lambda$. Next, since $Y_{010}=Y_{101}$, the action $v$ of $Y_0$ on $M_2$
is the same as the action of $Y_1$ on $M_0$ via these identities, and
this action squares to zero. Now
the central element $\gamma$ has to act zero. It acts on $M_1$ as 
$Y_{1210}+Y_{1202}$. Since $Y_{12}\colon M_0\to M_1$ is the isomorphism $u$, it follows that $Y_{10}+Y_{02}\colon M_1\to M_0$ acts
as zero. Therefore $Y_2$ acts as $-v$ on $M_1$. Finally, since
$Y_{212}=Y_{121}\colon M_0\to M_1$, we have $uv=-vu$. Thus we have the
following diagram of weight spaces for any module with central
character of type (2).

{\small
\begin{equation}\label{eq:3spaces} 
\vcenter{\xymatrix@R=20mm@C=10mm{M_2\ar[rr]^{Y_{1}}_I\ar@(l,u)[]^{Y_0}_v & 
& M_1\ar[dl]^{Y_{0}}_I\ar@(u,r)[]^{Y_2}_{-\!v\!} \\
&M_0\ar[ul]^{Y_{2}}_{u}\ar@(dr,dl)[]^{Y_1}_v}}
\end{equation}
}%
Let $U$ be the vector space obtained by identifying $M_0$, $M_1$ and
$M_2$ using the above identity maps. Then $U$ is a representation of
the algebra
\[ B=\tilde k\langle u,v\rangle/( u^2-\lambda,v^2,uv+vu), \]
Conversely, given any representation $U$ of the algebra $B$, we obtain
an $A$-module with the correct central character via the 
diagram~\eqref{eq:3spaces}, using three copies of $U$.

It follows that we have $A/\mfm A\cong \Mat_3(B)$ via the isomorphism
\[ Y_0 \mapsto \left(\begin{smallmatrix}
0&1&0\\ 0&0&0 \\ 0&0&v 
\end{smallmatrix}\right)\qquad 
Y_1 \mapsto \left(\begin{smallmatrix}
v&0&0\\ 0&0&1 \\ 0&0&0 
\end{smallmatrix}\right)\qquad 
Y_2 \mapsto \left(\begin{smallmatrix}
0&0&0\\ 0&-v&0 \\ u&0&0 
\end{smallmatrix}\right) \]
and $B$ is the basic algebra of $A/\mfm A$.

There are now
three cases for the structure of $B$ and hence the representation
theory of the corresponding block, according to
the field $\tilde k$. First suppose that $\lambda$ is a square in
$\tilde k$.
If the characteristic is not two, then the block has
two simple modules of dimension three,  corresponding to the
two square roots of $\lambda$, as follows:
\[ \xymatrix{M_2\ar[r]_{Y_1}& M_1\ar[r]_{Y_0}&
    M_0\ar[r]_{Y_2}^{\pm\sqrt\lambda}& M_2}. \]
Here, the maps not labelled by a 
scalar act as the identity.
The two distinct simple modules $M$ and $M'$ extend each
other to make projective indecomposable $A/\mfm A$-modules of dimension six over $\tilde
k$. These are as follows.
\[ \xymatrix{M'_2\ar[r]^{Y_1}\ar[d]_{Y_0}&
M'_1\ar[r]^{Y_0}\ar[d]_{Y_2}^{-1}&
M'_0\ar[r]^{Y_2}_{\mp\sqrt\lambda}\ar[d]_{Y_1} & M'_2\ar[d]_{Y_0} \\
M_2\ar[r]_{Y_1} &M_1\ar[r]_{Y_0}& M_0\ar[r]_{Y_2}^{\pm\sqrt\lambda}&M_2} \]
with the left side identified with the right side. Here, all the six
spaces are isomorphic to $\tilde k$.

In characteristic two, there is only one simple module of dimension
three, corresponding to the one square root of $\lambda$, 
and its projective cover has length four and Loewy length three.

If $\lambda$ is not a square in $\tilde k$, in any characteristic, then there is
one simple module, of dimension three over $\tilde k(\sqrt\lambda)$,
and it extends itself to form the projective indecomposable. In all three cases, the
block has dimension $4 \times 9 = 36$.\bigskip

This completes the verification of Conjecture~\ref{conj:multinomial}
in the case of $\mcN(\tilde A_2)$.

\section{\texorpdfstring{Example: The case $\tilde A_3$}
{Example: The case Ã₃}}

The nilCoxeter algebra $A$ of type $\tilde A_3$ has four generators,
$Y_0$, $Y_1$, $Y_2$, $Y_3$, subject to the relations
\begin{gather*} Y_0^2=Y_1^2=Y_2^2=Y_3^2=0,\quad Y_{010}=Y_{101},\quad
Y_{121}=Y_{212},\\ Y_{323}=Y_{232},\quad
Y_{030}=Y_{303},\quad Y_{02}=Y_{20},\quad Y_{13}=Y_{31}. 
\end{gather*}

The commutative subalgebra $C$ may be described as follows. The
quotient by the prime ideal $\mfp_0$ has a basis consisting of the
images of the $Z_v$ of Section~\ref{se:C} where
$v=(r_1,r_2,r_3,r_4)$ satisfies $r_1\le r_2\le r_3\le r_4$. Let 
$a=(-3,1,1,1)$, $b=(-1,-1,1,1)$, $c=(-1,-1,-1,3)$, $d=(-2,0,1,1)$,
$e=(-1,-1,0,2)$, $f=(-1,0,0,1)$. The corresponding elements of $C$ are 
\begin{align*}
Z_a&=Y_{1230}^3&Z_d&=Y_{21230}^2\\
Z_b&=Y_{2310}^2& Z_e&=Y_{23210}^2\\
Z_c&=Y_{3210}^3&Z_f&=Y_{123210}=Y_{321230} 
\end{align*}
The ring $C/\mfp_0$ is generated by the images of $Z_a,\dots,Z_f$, so
that it is isomorphic to
\[ k[Z_a,Z_b,Z_c,Z_d,Z_e,Z_f]/(Z_d^2-Z_aZ_b,
Z_e^2-Z_bZ_c,Z_f^4-Z_aZ_c,Z_bZ_f^2-Z_dZ_e). \]
We write $\alpha,\beta,\gamma,\delta,\ep,\phi$ for the
orbit sums $\hat Z_a,\dots,\hat Z_f$ under $\sym{4}$, satisfying the same
relations. 
\begin{align*} 
\alpha=\hat Z_a &=Y_{1230}^3+Y_{2301}^3+Y_{3012}^3+Y_{0123}^3
=(Y_{1230}+Y_{2301}+Y_{3012}+Y_{0123})^3\\
\beta=\hat Z_b
&=Y_{2310}^2+Y_{3021}^2+Y_{0132}^2+Y_{1203}^2+Y_{3120}^2+Y_{2031}^2\\
\gamma=\hat Z_c&=Y_{3210}^3+Y_{0321}^3+Y_{1032}^3+Y_{2103}^3
=(Y_{3210}+Y_{0321}+Y_{1032}+Y_{2103})^3\\
\delta=\hat Z_d&=Y_{21230}^2+Y_{32301}^2+Y_{03012}^2+Y_{10123}^2
+Y_{12320}^2+Y_{23031}^2+Y_{30102}^2+Y_{01213}^2\\
&\quad+Y_{12030}^2+Y_{23101}^2+Y_{30212}^2+Y_{01323}^2\\
\ep=\hat Z_e&=Y_{23210}^2+Y_{30321}^2+Y_{01032}^2+Y_{12103}^2
+Y_{32120}^2+Y_{03231}^2+Y_{10302}^2+Y_{21013}^2\\
&\quad+Y_{32010}^2+Y_{03121}^2+Y_{10232}^2+Y_{21303}^2\\
\phi=\hat  Z_f&=Y_{123210}+Y_{230321}+Y_{301032}+Y_{012103}
+Y_{012321}+Y_{123032}+Y_{230103}+Y_{301210}\\
&\quad+Y_{212030}+Y_{323101}+Y_{030212}+Y_{101323}
\end{align*}
Then we have
\[ Z(A)=R=k[\alpha,\beta,\gamma,\delta,\ep,\phi]/
(\delta^2-\alpha\beta,\ep^2-\beta\gamma,\phi^4-\alpha\gamma,
\beta\phi^2-\delta\ep). \]
This is a normal integral domain of Krull dimension three. By
Theorem~\ref{th:R}, it can be thought
of as the ring 
\[ k[X_1^4,X_2^2,X_3^4,X_1^2X_2,X_2X_3^2,X_1X_3]
=k[X_1,X_2,X_3]^{\mu_4} . \]

The dense open subset of $\Spec(R)$ giving the Azumaya locus of $A$ is
the set of maximal ideals $\mfm$ such that all the variables
have non-zero images in $R/\mfm$. For these central characters, there is one
simple module of dimension $24$, and we have $A/\mfm
A\cong\Mat_{24}(\tilde k)$. At the other extreme, the maximal ideal
generated by all the variables corresponds to the one dimensional
trivial module. Between these extremes, the possible stabilisers of
the weight spaces are the subgroups $\sym{3}$, 
$\sym{2}\times\sym{2}$ and $\sym{2}$. In the following table, the
eight possibilities are determined by which ones of $\alpha$, $\beta$
and $\gamma$ have non-zero image in $R/\mfm$; we write $*$ to denotes
non-zero image in $R/\mfm$ and $0$ to denote zero image in $R/\mfm$:
\[ \begin{array}{|c|c|cccccc|}
\hline
\text{Stabiliser}&\text{\# Wt spaces}&\alpha&\beta&\gamma&\delta&\ep&\phi\\ \hline
\sym{4}&1&0&0&0&0&0&0\\
1\times\sym{3}&4&*&0&0&0&0&0\\
\sym{3}\times1&4&0&0&*&0&0&0\\
\sym{2}\times\sym{2}&6&0&*&0&0&0&0\\
1\times1\times\sym{2}&12&*&*&0&*&0&0\\
1\times\sym{2}\times 1&12&*&0&*&0&0&*\\
\sym{2}\times 1\times 1&12&0&*&*&0&*&0\\
1\times1\times1\times1&24&*&*&*&*&*&*\\ \hline
\end{array} \]
We shall investigate in detail the cases $\sym{4}$, $1\times \sym{3}$,
$\sym{2}\times\sym{2}$ and $1\times 1 \times \sym{2}$, and trivial
stabiliser. Cyclic 
permutations of the generators of $A$ then deal with the cases
$\sym{3}\times 1$, $1\times \sym{2} \times 1$,  $\sym{2}\times
1\times 1$. These cyclic permutations fix $C$ setwise, while
cyclically permuting the coordinates in vector notation. This changes
the choice of $\mfp_0$, but preserves the algebra structure of $A/\mfp
A$ and hence the structure of the basic algebra. 
For simplicity we shall assume that $k$ is algebraically closed.\bigskip

In the case with 
trivial stabiliser, there is not much to say, because $A/\mfm
A\cong\Mat_{24}(k)$. In the case with stabiliser $\sym{4}$, $A/\mfm A$ is a
local algebra with just one simple module, the trivial module. At
least in small characteristics (via {\sc Magma}), and probably in all characteristics,
the dimension of $A/\mfm A$ is $1440$, with radical layers of dimension
\[ 1,\ 4,\ 10,\ 20,\ 34,\ 52,\ 73,\ 96,\ 119,\ 140,\ 156,\ 164,\ 160,\
  142,\ 113,\ 78,\ 46,\ 22,\ 8,\ 2, \]
 but the
radical layers are not equal to the socle layers, which have dimension 
(in reverse order to match the above)
\[  1 ,\ 4,\ 10,\ 20,\ 33,\ 48,\ 64,\ 80,\ 96,\ 112,\ 128,\ 
144,\ 155,\ 156,\ 142,\ 112,\ 75,\ 40,\ 16,\ 4. \]  
We shall not say more about its
structure.\bigskip

Let us next consider the case with stabiliser $1\times\sym{3}$, where
$\alpha=\hat Z_a$ has non-zero image in $R/\mfm$, and the remaining variables
have zero image. Then 
\[ \mfm=(\hat Z_a-\lambda,\hat Z_b,\hat Z_c,\hat Z_d,\hat
  Z_e,\hat Z_f)\]
for some non-zero constant $\lambda\in k$. Let $M_i$ be the weight space stabilised by the
conjugate of $\sym{3}$ in $\sym{4}$ that stabilises $i$. Then $M_1$
corresponds to $C/\mfp_0$.  We have
$Y_1Z_a=0$, so $Y_1M_1=0$. By Lemma~\ref{le:permute-weights}, we have
$Y_2M_1\subseteq M_1$, $Y_3M_1\subseteq M_1$, and 
$Y_0M_1\subseteq M_0$. We obtain a diagram as follows, by 
using the braid relations to relate the self-maps at the four vertices.
\[ \vcenter{\xymatrix@=20mm{
M_3\ar[r]^{Y_{2}}_I\ar@(dl,ul)[]^{Y_1}_{v}\ar@(ul,ur)[]^{Y_0}_{w}  
& M_2\ar[d]^{Y_{1}}_I\ar@(ul,ur)[]^{Y_0}_{w}\ar@(ur,dr)[]^{Y_3}_{x} \\
M_4\ar[u]^{Y_3}_{u}\ar@(dr,dl)[]^{Y_2}_{v}\ar@(dl,ul)[]^{Y_1}_{w}
&M_1\ar[l]^{Y_{0}}_{I}\ar@(ur,dr)[]^{Y_3}_{x}\ar@(dr,dl)[]^{Y_2}_{v}}} \]
Then $u^3$ acts as multiplication by the scalar $\lambda$, so $u$ is an isomorphism. By the braid relations we also have $uv=xu$, $uw=vu$, and $ux=wu$.
Now the action of $\hat Z_d$ on $M_1$
has to be zero. All terms in $\hat Z_d$ except for $Y_{21230}^2$,
$Y_{12320}^2$ and $Y_{12030}^2$
already act as zero, and these terms acts as $(vu)^2+(uv)^2+(wu)^2=(vx+xw+wv)u^2$. Since $u$ is
an isomorphism, we deduce that $vx+xw+wv=0$. Similarly, $\hat Z_f$ acts on $M_1$ as $uvw+uxv+uwx$, so
we have $vw+xv+wx=0$. Since $\hat Z_a$ acts invertibly, $\hat Z_d$
and $\hat Z_f$ acting as zero already force $\hat Z_b$ and $\hat Z_c$
to act as zero. In the algebra
\begin{multline*} 
B=k\langle u,v,w,x\rangle
/(u^3-\lambda,v^2,w^2,x^2, uv-xu,uw-vu,ux-wu,\\
vx+xw+wv,vw+xv+wx) 
\end{multline*}
the action of $\hat Z_e$ is already zero. The
ideal in $B$ generated by $v$, $w$ and $x$ is nilpotent. It follows that for
a simple $A/\mfm A$-module, $v$, $w$ and $x$ act as zero. So the
simple $A/\mfm A$-modules are four dimensional.

Thus $B$ is the basic algebra of $A/\mfm A$, and $A/\mfm A\cong
\Mat_4(B)$. As long as the characteristic is not three, the algebra $B$ has dimension $36$, is self-injective, with three
simple modules corresponding to the cube roots of $\lambda$. The
radical layers of the
projective indecomposable $B$-modules have dimensions $1$, $3$, $4$,
$3$, $1$. Thus $A/\mfm A$ has dimension $16 \times 36 = 576=4!^2$.\bigskip

Next we consider the case with $\sym{2}\times\sym{2}$
stabiliser. This corresponds to a maximal ideal of the form
\[ \mfm = (\hat Z_a,\hat Z_b-\lambda,\hat Z_c,\hat Z_d,\hat Z_e,\hat Z_f)\] with $0\neq \lambda \in k$. The diagram of $6$ weight spaces (corresponding to the row standard tabloids of shape $(2,2)$) labelled with unordered pairs in this case is as follows. 
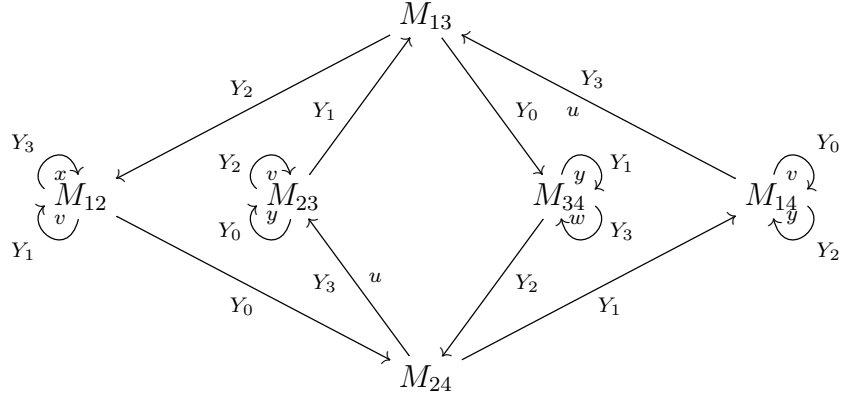
\[ \xymatrix@R=18mm@C=8mm{&&&M_{13}\ar[dlll]_{Y_2}\ar[dr]^(.6){Y_0}\\
M_{12}\ar[drrr]_{Y_0}\ar@(d,l)[]^{Y_1}_v\ar@(l,u)^{Y_3}_x&&
M_{23}\ar[ur]^(.4){Y_1}\ar@(d,l)^(.6){Y_0}_y\ar@(l,u)^(.4){Y_2}_v&&
M_{34}\ar[dl]^(.4){Y_2}\ar@(u,r)^(.6){Y_1}_y\ar@(r,d)^(.4){Y_3}_w&&
M_{14}\ar[ulll]_(.55){Y_3}^(.55)u\ar@(u,r)^{Y_0}_v\ar@(r,d)^{Y_2}_y\\
&&&M_{24}\ar[ul]^(.6){Y_3}_u\ar[urrr]_{Y_1}} \]
with $u^2=\lambda$, $v^2=0$, $w^2=0$, $x^2=0$, $y^2=0$, $uv=wu$, $vu=uw$, $xu=uy$, $yu=ux$, $xv=vx$, $xw=wx$,
$yv=vy$, $yw=wy$. The fact that $\hat Z_d$, $\hat Z_e$ and $\hat
Z_f$ act as zero on $M_{12}$ imply that $vw+wv=0$, $xy+yx=0$ and
$(v+w)(x+y)=0$, and also imply that $\hat Z_a$ and $\hat Z_c$ act as
zero. So we have $A/\mfm A\cong \Mat_6(B)$, where 
\begin{multline*} 
B = k\langle  u,v,w,x,y\rangle/(u^2-\lambda,v^2,w^2,x^2,y^2,
uv-wu,uw-vu,ux-yu,uy-xu,\\
vx-xv,wx-xw,vy-yv,wy-yw,vw+wv,xy+yx,(v+w)(x+y)) 
\end{multline*}
The generators $v$, $w$, $x$ and $y$ generate a nilpotent ideal in
$B$, so they act as zero on any simple $A/\mfm A$-module. As long as the characteristic is not two, there are
two simple $B$-modules, and therefore there are two six-dimensional simple $A$-modules.
The radical layers of the projective
indecomposable $B$-modules have dimensions $1$, $4$, $5$, $2$. 
Thus $A/\mfm A$ has dimension $2\times 12 \times 36 = 864 > 4!^2$ and it is not self-injective.

This and the $\sym{4}$ stabiliser case are the
only non self-injective cases in type $\tilde A_3$.\bigskip

Next, the case with stabiliser $1\times 1 \times \sym{2}$. This
corresponds to a maximal ideal of the form
\[ \mfm=(\hat Z_a-\lambda,\hat Z_b-\mu,\hat Z_c,\hat
  Z_d-\nu,\hat Z_e,\hat Z_f) \]
with $\nu^2=\lambda\mu$. We obtain  the
following diagram by using the braid relations to relate the self-maps
at the vertices. In  this diagram of 12 weight spaces (corresponding to the row standard tabloids of shape $(1,1,2)$) labelled by ordered pairs, the left and right sides are identified.
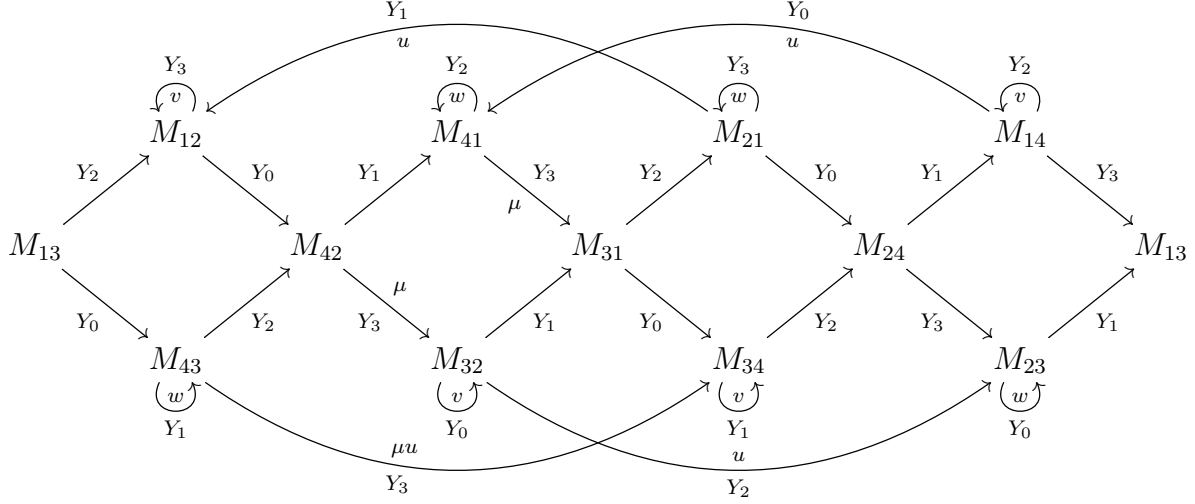
\[ \xymatrix@=9mm{ &
M_{12}\ar[dr]^{Y_0}\ar@(ur,ul)[]_{Y_3}^v&&
M_{41}\ar[dr]^{Y_3}_\mu\ar@(ur,ul)[]_{Y_2}^w&&
M_{21}\ar[dr]^{Y_0}\ar@(ur,ul)[]_{Y_3}^w\ar@/_8ex/[llll]_(.6){Y_1}^(.6)u&&
M_{14}\ar[dr]^{Y_3}\ar@(ur,ul)[]_{Y_2}^v\ar@/_8ex/[llll]_(.4){Y_0}^(.4)u\\
M_{13}\ar[ur]^{Y_2}\ar[dr]_{Y_0}&&
M_{42}\ar[ur]^{Y_1}\ar[dr]_{Y_3}^\mu&&
M_{31}\ar[ur]^{Y_2}\ar[dr]_{Y_0}&&
M_{24}\ar[ur]^{Y_1}\ar[dr]_{Y_3}&&
M_{13}\\ &
M_{43}\ar[ur]_{Y_2}\ar@(dl,dr)[]_{Y_1}^w\ar@/_8ex/[rrrr]_(.4){Y_3}^(.4){\mu u}&&
M_{32}\ar[ur]_{Y_1}\ar@(dl,dr)[]_{Y_0}^v\ar@/_8ex/[rrrr]_{Y_2}^u&&
M_{34}\ar[ur]_{Y_2}\ar@(dl,dr)[]_{Y_1}^v&&
M_{23}\ar[ur]_{Y_1}\ar@(dl,dr)[]_{Y_0}^w} \]
with $u^2=\nu/\mu^2$, $uv=wu$, $uw=vu$. The vanishing of $\hat Z_f$
then forces $w=-v$, so we have $uv=-vu$. If we add this relation, then
$\hat Z_c$ and $\hat Z_e$ vanish and we are done. Thus the basic algebra is
\[ B=k\langle u,v\rangle/(u^2-\nu/\mu^2, v^2, uv+vu) \]
of dimension four. We have $A/\mfm A\cong \Mat_{12}(B)$, of dimension
$4\times 12^2=576=24^2$.\bigskip

This completes the verification of Conjecture~\ref{conj:multinomial}
in the case of $\mcN(\tilde A_3)$.

\bibliographystyle{amsplain}
\bibliography{../repcoh}

\end{document}